\documentclass{amsart}
\usepackage[english]{babel}
\usepackage{amsmath,amssymb,amsthm,amscd,comment}
\usepackage{graphicx,subfigure}
\usepackage{color}
\usepackage{pb-diagram} 
\usepackage{tabularx}
\usepackage{multirow}
\usepackage{longtable}
\usepackage{enumitem}
\usepackage[foot]{amsaddr}
\usepackage{mathtools}

\newtheorem*{theorem*}{Theorem}
\newtheorem{theorem}{Theorem}[section]
\newtheorem{proposition}[theorem]{Proposition}
\newtheorem{corollary}[theorem]{Corollary}

\newtheorem{lemma}[theorem]{Lemma}
\newtheorem{definition}[theorem]{Definition}

\theoremstyle{remark} 
\newtheorem{remark}[theorem]{Remark}

\numberwithin{equation}{section}

\newcommand{\al}{\alpha}

\newcommand{\ep}{\varepsilon}

\newcommand{\te}{\theta}

\newcommand{\cA}{{\mathcal A}}

\newcommand{\cD}{{\mathcal D}}
\newcommand{\cE}{{\mathcal E}}
\newcommand{\cF}{{\mathcal F}}
\newcommand{\cG}{{\mathcal G}}

\newcommand{\cO}{{\mathcal O}}

\newcommand{\cR}{{\mathcal R}}
\newcommand{\cS}{{\mathcal S}}

\newcommand{\RR}{{\mathbb R}}
\newcommand{\CC}{{\mathbb C}}

\newcommand{\TT}{{\mathbb T}}
\newcommand{\ZZ}{{\mathbb Z}}
\newcommand{\NN}{{\mathbb N}}
\newcommand{\QQ}{{\mathbb Q}}

\newcommand{\abs}[1]{|{#1}|}
\newcommand{\Abs}[1]{\left|{#1}\right|}
\newcommand{\norm}[1]{\|{#1}\|}
\newcommand{\Norm}[1]{\left\|{#1}\right\|}
\newcommand{\aver}[1]{\langle{#1}\rangle}
\newcommand{\dist}[1]{\mathrm{dist}\left({#1}\right)}

\def\epsilon{\varepsilon}

\def\ii{\mathrm{i}}
\def\ee{\mathrm{e}}

\def\im{\mathrm{Im}}
\def\id{\mathrm{id}}

\def\dif{ {\mbox{\rm d}} }
\def\Dif{ {\mbox{\rm D}} }

\def\pd{ \partial }

\def\Sin{S}
\def\Cos{C}

\def\Anal{{\cA}}

\def\CLipO{C_0^\sLip}
\def\CLipI{C_1^\sLip}
\def\CLipII{C_2^\sLip}
\def\CLipIII{C_3^\sLip}
\def\CLiph{C_4^\sLip}
\def\CLipIV{C_5^\sLip}
\def\CLipEE{C_6^\sLip}
\def\CLipV{C_7^\sLip}
\def\CLipVI{C_8^\sLip}
\def\CLipVII{C_9^\sLip}

\def\GI{\mathfrak{G}_1}
\def\GII{\mathfrak{G}_2}
\def\HI{\mathfrak{H}_1}
\def\HII{\mathfrak{H}_2}
\def\HIII{\mathfrak{H}_3}

\def\TI{\mathfrak{T}_1}
\def\TII{\mathfrak{T}_2}

\def\TIV{\mathfrak{T}_4}
\def\PI{\mathfrak{P}_1}
\def\PII{\mathfrak{P}_2}
\def\PIII{\mathfrak{P}_3}
\def\PIV{\mathfrak{P}_4}
\def\PV{\mathfrak{P}_5}
\def\PVI{\mathfrak{P}_6}

\def\al{{\alpha}}
\def\ep{{\varepsilon}}
\def\map{{f}}
\def\conj{{h}}
\def\rot{{\theta}}

\def\Hom{{\mathrm{Hom}}}
\newcommand{\sigh}{{\sigma_1}}
\newcommand{\sighi}{{\sigma_2}}
\newcommand{\sighxx}{{\sigma_3}}
\newcommand{\sighb}{{\sigma_b}}
\newcommand{\betah}{{\beta_0}}
\newcommand{\betaDh}{{\beta_1}}
\newcommand{\betaa}{{\beta_2}}
\newcommand{\Difeo}[1]{{\mathcal{A}(\TT_{#1})}}
\newcommand{\Per}[1]{{\mathrm{Per}(\TT_{#1})}}
\newcommand{\Lip}[1]{{\mathrm{Lip}_{#1}}}
\newcommand{\sLip}{\mathrm{Lip}}
\newcommand{\lip}[1]{{\mathrm{lip}_{#1}}}
\newcommand{\meas}{\mathrm{Leb}}

\allowdisplaybreaks

\begin{document}

\title{
Effective bounds for the measure of rotations
}

\date{\today}

\author{Jordi-Llu\'{\i}s Figueras$^\clubsuit$}
\address[$\clubsuit$]{Department of Mathematics, Uppsala University,
Box 480, 751 06 Uppsala, Sweden}
\email{figueras@math.uu.se}

\author{Alex Haro$^\diamondsuit$}
\address[$\diamondsuit$]{Departament de Matem\`atiques i Inform\`atica, Universitat de Barcelona,
Gran Via 585, 08007 Barcelona, Spain.}
\email{alex@maia.ub.es}

\author{Alejandro Luque$^\spadesuit$}
\address[$\spadesuit$]{Department of Mathematics, Uppsala University, 
Box 480, 751 06 Uppsala, Sweden}
\email{alejandro.luque@math.uu.se}

\begin{abstract}
A fundamental question in Dynamical Systems is to identify regions of
phase/parameter space satisfying a given property (stability, linearization,
etc).  Given a family of analytic circle diffeomorphisms depending on a
parameter, we obtain effective (almost optimal) lower bounds
of the Lebesgue measure of the set of parameters
that are conjugated to a rigid rotation.
We estimate this measure using an a-posteriori KAM
scheme that relies on quantitative conditions that
are checkable using computer-assistance. We carefully describe
how the 
hypotheses in our theorems are reduced to a finite number of
computations, and apply our methodology to the case of the
Arnold family. Hence we show that obtaining non-asymptotic lower bounds for
the applicability of KAM theorems is a feasible task provided one has an
a-posteriori theorem to characterize the problem.  Finally, 
as a direct corollary, we produce explicit asymptotic
estimates in the so called local reduction setting (\`a la Arnold) which are
valid for a global set of rotations.  
\end{abstract}

\maketitle



\tableofcontents

\section{Introduction}\label{sec:intro}

One of the most important problems in Hamiltonian Mechanics, 
and Dynami\-cal Systems in general, is to
identify stability (and instability) regions in
phase and parameter space. The question of stability goes 
back to the classical works
of outstanding mathematicians 
during the XVIII$^\mathrm{th}$ to early XX$^\mathrm{th}$
centuries, such as 
Laplace, 
Kovalevskaya, 
Poincar\'e, 
Lyapunov, 
or
Birkhoff, 
who 
thought about
the problem and obtained important
results in this direction~\cite{Birkhoff66,Dumas14,Poincare87b,Stoker55}.
During the mid 1950's the birth of KAM theory \cite{Arnold63a,Kolmogorov54,Moser62}
gave certain hope in the characterization of stable motions, not only by proving the existence
of quasi-periodic solutions but also revealing that they are present in
regions of positive measure in phase space
\cite{Lazutkin74,Neishtadt81,Poschel82}.
Notwithstanding the formidable impact of ideas and results produced
in the perturbative context \cite{ArnoldKN06,BroerHS96,Zehnder75}, 
even in recent works \cite{
ChierchiaP11,
EliassonFK15,
EncisoP15,
Fejoz04},
the theory was for long time attributed to be
seriously limited in the study of concrete and realistic problems.
On the one hand, the size of the
perturbation admitted in 
early KAM theorems was dramatically 
small\footnote{Giving a satisfactory account about
skeptic criticisms in this direction 
is far from the scope of this paper, but
we refer to \cite{CellettiC07,Dumas14} for illuminating details and references.},
and on the other hand, as far as the authors know, the only knowledge about
measure estimates in phase space is just asymptotic:
the union of the surviving invariant tori has relative measure of order at least 
$1-\sqrt{\ep}$, where $\ep$ is the perturbation parameter
(see the above references and also \cite{BiascoC15,MedvedevNT15}, 
where secondary tori are also considered).

The aim of this paper is to show that
the task of obtaining effective bounds for the measure
of quasi-periodic solutions in phase space is feasible, 
and to provide full details in the setting of conjugacy to
rotation of (analytic) circle diffeomorphisms.
The study of maps of the circle to itself is one of the
most fundamental dynamical systems, and many years after Poincar\'e
raised the question of comparing the dynamics of a circle 
homeomorphism with a rigid rotation,
it is perhaps the
problem where the global effect of small divisors is best
understood.
This problem was approached by Arnold himself in~\cite{Arnold61}
who obtained mild conditions for conjugacy to rotation
in a perturbative setting (known as \emph{local reduction theorem}
in this context), showing also that the existence
and smoothness of such conjugacy is closely connected with
the existence of an absolutely continuous invariant measure.
The first global results (known as \emph{global reduction theorems})
were obtained in~\cite{Herman79} and
extended later in~\cite{Yoccoz84}.
Sharp estimates on finite regularity were
investigated along different works~\cite{KatznelsonO89,
KhaninS87,
SinaiK89} and finally obtained in~\cite{KhaninT09}.

Before going into the details, we think it is convenient 
to give an
overview on the progress to effectively apply KAM theory in particular systems.
With the advent of computers and new developed methodologies, the applicability
of the theory has been manifested in the pioneering works
\cite{CellettiC88,LlaveR91} and in applications to Celestial Mecha\-nics
\cite{CellettiC97,CellettiC07,Locatelli98,LocatelliG00}.  A recent methodology
has been proposed in \cite{FiguerasHL17}, based in an a-posteriori KAM theorem
with quantitative and sharp explicit hypotheses. To check the hypotheses of
the theorem, a major difficulty is to control the analytic norm of some
complicated functions defined on the torus. This is done using fast Fourier
transform (with interval arithmetics) and carrying out an accurate control of
the discretization error.
The methodology has been applied to low dimensional problems obtaining almost
optimal results.  But after the previous mentioned works, the most important
question remained open, that is:
\begin{quote}
\emph{Given a particular system with non-perturbative
parameters, and given a particular region of interest in
phase/parameter space,
what is the abundance of quasiperiodic smooth solutions in that region?}
\end{quote}

From the perspective of characterizing the Lebesgue measure of such
solutions, the above question
is a global version of the perturbative (asymptotic) estimates
for the measure of KAM tori~\cite{Arnold63a,Lazutkin74,Neishtadt81,Poschel82}.
In contrast, an analogous global question regarding the topological
characterization of instability was formulated by Herman~\cite{Herman98} in terms
of non-wandering sets.
Although non-perturbative and global questions are of the highest
interest in the study of a dynamical systems, it is not
surprising that they are not often explicitly formulated in the literature
(with exception of~\cite{Birkhoff66,Stoker55} and some numerical studies,
e.g. \cite{Laskar89,RobutelL01,SimoT08}),
since the analytical approaches to the problem were limited by using
asymptotic estimates in the perturbation parameter.
Then, the work presented in this paper not only is valuable for the fact
that it provides a novel tool to use in KAM-like schemes, but also
enlarges our vision about how stability can be effectively measured.

In this paper the above question is directly formulated in the context of circle
maps. Given any family $\al \in A \rightarrow f_\al$
of analytic circle diffeomorphisms, we answer the following problem:

\begin{quote}
\emph{Obtain (almost optimal) lower bounds for the measure of parameters $\al \in A$
such that the map $f_{\al}$ is analytically conjugated to a rigid rotation.}
\end{quote}

This question was considered by Arnold in~\cite{Arnold61} (following
Poincar\'e's problem on the study of the rotation number as a function
on the space of mappings)
for the paradigmatic
example 
\begin{equation}\label{eq:AMAP}
\al \in [0,1] \longmapsto f_{\al,\ep}(x) = x + \al + \frac{\ep}{2\pi} \sin (2\pi)\,,
\end{equation}
where $|\ep|<1$ is a fixed parameter.
Denoting the rotation number of this family as $\rho_\ep : \al \in [0,1]
\mapsto \rho(f_{\al,\ep})$ and introducing the set $K_\ep = [0,1]\backslash
\mathrm{Int} (\rho_\ep^{-1}(\QQ))$, he was able to prove that
$\meas(K_\ep)\rightarrow 1$ for $|\ep|\rightarrow 0$, where $\meas(\cdot)$
stands for the Lebesgue measure.  As global results for $0 < |\ep|<1$, Herman
proved in~\cite{Herman79} 
that $K_\ep$ is a Cantor set, $K_\ep \cap \rho_\ep^{-1}(\QQ)$ is a countable set
dense in $K_\ep$, and $\rho_\ep^{-1}(p/q)$ is an interval with non-empty
interior for every $p/q \in \QQ$.  Still, Herman himself proved
in~\cite{Herman79b} that $\meas(K_\ep)>0$, but no quantitative estimates for
this measure are known.  A major difficulty is that, in the light of the
previous properties, $\rho_\ep$ is not a $C^1$ function ($\rho_\ep'$ blows up
in a dense set of points of $K_\ep$).  To deal with the task, we resort to an
a-posteriori KAM formulation of the problem that combines local and global
information and that we informally state as follows:

\begin{theorem*}
Let $A, B\subset \RR$ be open intervals and $\al \in A 
\mapsto
f_\al$ be a $C^3$-family
of analytic circle diffeomorphisms.
Let  $\theta \in B 
\mapsto
h_\theta$ be a Lipschitz family of
analytic circle diffeomorphisms and let 
$\theta \in B \mapsto \al(\theta) \in A$ be a Lipschitz function.
Under some mild and explicit conditions, if the family of
error functions $\theta\in B
\mapsto
e_\theta$ given by
\[
e_\theta(x)= f_{\al{(\theta)}}(h_\theta(x))-h_{\theta}(x+\theta)
\]
is small enough, then there exist a Cantor set 
$\Theta\subset B$ of positive measure, 
a Lipschitz family of
analytic circle diffeomorphisms $\theta \in \Theta 
\mapsto
\bar h_\theta$ and 
a Lipschitz function  
$\theta \in \Theta \mapsto \bar\al(\theta) \in A$
such that 
\[
    f_{\bar\al{(\theta)}}(\bar h_\theta(x))=\bar h_{\theta}(x+\theta).
\]
Moreover,  the measure of conjugacies in the space
of parameters, $\meas(\bar\al(\Theta))$, is controlled in terms of explicit estimates
that depend only on the initial objects 
and Diophantine properties defining $\Theta$.
\end{theorem*}

For the convenience of the reader, the above result
is presented in two parts.
In Section~\ref{ssec:conj} (Theorem~\ref{theo:KAM}) we present
an a-posteriori theorem for the existence of the conjugacy of
a fixed rotation number. Then,
in Section~\ref{ssec:conjL} (Theorem~\ref{theo:KAM:L})
we present an a-posteriori theorem to control the existence
and measure of conjugacies in a given interval of rotations.
Both results could be handled simultaneously, but
this splitting is useful when the time comes
to produce computer-assisted applications and also 
allows us to present the ideas in a self-consistent and more accessible way.

An important step to check the hypotheses of the theorem
is to put on the ground a solid 
theory based on Lindstedt series that allows us to perform all necessary 
computations effectively in a computer-assisted proof.
In Section~\ref{sec:hypo} we describe in detail, using analytic arguments, how
the hypotheses are thus reduced to a finite amount of computations
which can be implemented systematically. 
In particular, 
we explain how to control the norms
of Fourier-Taylor series
using suitable discretizations and taking into account the corresponding
remainders analytically. Indeed, the fact that Fourier-Taylor series can
be manipulated using fast Fourier methods is \emph{per se} a novel contribution in
this paper, so one can outperform the use of symbolic manipulators.

As an illustration of the effectiveness of the methodology we 
consider the
Arnold family~\eqref{eq:AMAP}.
For example, we prove that
\[
0.860748 < \meas(K_{0.25}) < 0.914161\,.
\]
The lower bound follows from the computer assisted application of
our main theorem and the upper bound is obtained by rigorous
computation of $p/q$-periodic orbits up to $q=20$.
Details, and further results, are given in Section~\ref{sec:example:Arnold}.

Regarding regularity,
we have constrained our result to the analytic case,
thus simplifying some intermediate estimates in the analytical part
exposed in this paper, but also, this is convenient 
for the control of the error of Fourier approximations in the computer-assisted application
of the method. This simplification only benefits the reader,
since the selected problem contains all the technical difficulties associated
to small divisors and illustrates very well the method proposed in this paper.

\section{Notation and elementary results}

We denote by $\TT = \RR/\ZZ$ the real circle.
We introduce a complex strip of $\TT$ of width $\rho>0$ as
\[
\TT_{\rho}= 
\left\{x \in \CC / \ZZ \ : \ \abs{\im{\,x}} < \rho \right\}\,,
\]
denote by $\bar \TT_{\rho}$ its closure, and by $\partial \TT_{\rho}=
\{|\im{\,x}|=\rho\}$ its boundary.

We denote by $\Per{\rho}$ the Banach space of periodic
continuous functions $f: \bar \TT_\rho \rightarrow \CC$, holomorphic in
$\TT_\rho$ and such that $f(\TT) \subset \RR$, endowed with the analytic norm
\[
\norm{f}_\rho \textcolor{blue}{:}= \sup_{x \in \TT_\rho} \abs{f(x)} \, .
\]
We denote the Fourier series of a periodic function $f$ as
\[
f (x) = \sum_{k \in \ZZ} \hat f_k \ee^{2\pi \ii x}
\]
and $\aver{f}:=\hat f_0$ stands
for the average.

In this paper we consider circle maps in the affine space
\[
\Difeo{\rho}=\left\{ \map \in \Hom(\TT)\,,\, f-{\rm id}\in
\Per{\rho}\right\} \,.
\]
Given an open set $A \subset \RR$ and a family of maps $\al \in A \rightarrow
\map_\al \in \Difeo{\rho}$, we consider the norms
\[
\norm{\partial_{x,\al}^{i,j} \map}_{A,\rho} \textcolor{blue}{:}= \sup_{\al \in
A} \norm{\partial_{x,\al}^{i,j}\map_\al}_\rho \, ,
\]
provided $\partial_{x,\al}^{i,j} \map_\al \in \Per{\rho}$ for every
$\al \in A$.

Finally, we introduce some useful notation regarding solutions
of cohomological equations. Given a zero-average function $\eta$,
we consider the linear difference equation
\begin{equation}\label{eq:coho0}
\varphi(x+\rot) -
\varphi(x)
= \eta(x)\,.
\end{equation}
To ensure regularity of the solutions of this equation, some
arithmetic conditions on the rotation number are required.
Given $\gamma>0$ and $\tau\geq 1$, the set of $(\gamma,\tau)$-Diophantine
numbers is given by
\[
\cD(\gamma,\tau):=\{
\theta \in \RR 
\, : \,
|q\theta-p|\geq \gamma |q|^{-\tau}
\, , \,
\forall (p,q)\in \ZZ^2
\, , \,
q\neq 0
\}.
\]
For the scope of this paper, the following 
classic lemma is enough.

\begin{lemma}[R\"usmann estimates \cite{Russmann76a}]\label{lem:Russ}
Let $\rot \in \cD(\gamma,\tau)$.
Then, for every zero-average function $\eta \in \Per{\rho}$,
there exists a unique zero-average solution $\cR \eta$ of equation
\eqref{eq:coho0} such that, for any
$0<\delta \leq \rho$, we have 
$\cR \eta \in \Per{\rho-\delta}$ and
\[
\norm{\cR \eta}_{\rho-\delta} \leq \frac{c_R \norm{\eta}_\rho}{\gamma
\delta^\tau}\,, \qquad \mbox{with} \qquad 
c_R=\frac{\sqrt{\zeta(2,2^\tau)\Gamma(2\tau+1)}}{2(2\pi)^\tau}\,,
\]
where $\Gamma$ and $\zeta$ are the Gamma and
Hurwitz zeta functions, respectively.
\end{lemma}

Assume that $f$ is a function defined in $\Theta \subset \RR$ (not necessarily
an interval) and taking values in $\CC$. We say that $f$ is Lipschitz in
$\Theta$ if
\[
\Lip{\Theta}(f)\textcolor{blue}{:}=\sup_{\substack{\theta_1,\theta_2\in \Theta
\\ \theta_1 \neq \theta_2}}
\frac{|f(\theta_2)-f(\theta_1)|}{|\theta_2-\theta_1|} < \infty\,.
\]
We define $\norm{f}_\Theta \textcolor{blue}{:}= \sup_{\theta \in \Theta}
|f(\theta)|$. Similarly, if we take a family $\theta \in \Theta \mapsto
f_\theta \in \Difeo{\rho}$, we extend the previous notations as
\[
\Lip{\Theta,\rho}(f)\textcolor{blue}{:}=\sup_{\substack{\theta_1,\theta_2\in
\Theta \\ \theta_1 \neq \theta_2}}
\frac{\norm{f_{\theta_2}-f_{\theta_1}}_\rho}{|\theta_2-\theta_1|}\,, \qquad
\norm{f}_{\Theta,\rho} \textcolor{blue}{:}= \sup_{\theta \in \Theta}
\norm{f_\theta}_\rho.
\]
Finally, we say that a function $f$ is Lipschitz from below in $\Theta$ if
\[
\lip{\Theta}(f)\textcolor{blue}{:}=\inf_{\substack{\theta_1,\theta_2\in \Theta
\\ \theta_1 \neq \theta_2}}
\frac{|f(\theta_2)-f(\theta_1)|}{|\theta_2-\theta_1|} > 0\,.
\]

To obtain several estimates required in the paper, we will
use the following elementary properties.

\begin{lemma}\label{lem:Lipschitz}
Assume that $f,g$ are Lipschitz functions defined in $\Theta \subset \RR$
and taking values in $\CC$. Then
\begin{itemize}
\item [$\PI$] $\Lip{\Theta}(f+g) \leq \Lip{\Theta}(f)+\Lip{\Theta}(g)$,
\item [$\PII$] $\Lip{\Theta}(fg) \leq
\Lip{\Theta}(f)\norm{g}_\Theta+\norm{f}_\Theta \Lip{\Theta}(g)$,
\item [$\PIII$] $\Lip{\Theta}(f/g) \leq
\Lip{\Theta}(f)\norm{1/g}_\Theta+\norm{f}_\Theta \norm{1/g}_\Theta^2
\Lip{\Theta}(g)$.
\end{itemize}
Assume $0<\rho\leq \hat\rho$, $\delta > 0$, $\rho-\delta > 0$, 
and that we have families $\theta \in \Theta \mapsto f_\theta \in
\Difeo{\hat\rho}$ and $\theta \in \Theta \mapsto g_\theta\in \Difeo{\rho}$ such
that $g_\theta(\bar \TT_\rho) \subseteq \bar \TT_{\hat \rho}$. Then
\begin{itemize}
\item [$\PIV$] $\Lip{\Theta,\rho}(f \circ g) \leq
\Lip{\Theta,\hat\rho}(f)+\norm{\partial_xf}_{\Theta,\hat\rho}
\Lip{\Theta,\rho}(g)$,
\item [$\PV$]
$\Lip{\Theta,\rho-\delta}(\partial_x g) \leq
\frac{1}{\delta}\Lip{\Theta,\rho}(g)$.
\end{itemize}
Assume that we have a family
$\theta \in \Theta \mapsto g_\theta\in \Difeo{\rho}$,
where $\Theta \subset \RR \cap \cD(\gamma,\tau)$. Then,
if we denote $\cR_\theta g_\theta=f_\theta$ the zero-average
solution of $f_\theta(x+\rot) -
f_\theta(x)= g_\theta(x)$ with $\theta \in \Theta$, we have
\begin{itemize}
\item [$\PVI$] $\Lip{\Theta,\rho-\delta}(\cR g) \leq
\frac{c_R}{\gamma\delta^{\tau}} \Lip{\Theta,\rho}(g)
+ \frac{\hat c_R}{\gamma^2 \delta^{2\tau+1}} \norm{g}_{\Theta,\rho}$,
where
\[
\hat c_R = \tau^{-2\tau} (2\tau+1)^{2\tau+1} c_R^2\,.
\]
\end{itemize}
\end{lemma}

\section{An a-posteriori theorem for a single conjugation}\label{ssec:conj}

Given a map $f \in \Difeo{\hat \rho}$ with
rotation number $\theta \in \cD(\gamma,\tau)$, it is
well-known that there exists an analytic circle diffeomorphism that conjugates
$f$ to a rigid rotation of angle $\theta$
(see~\cite{Arnold61,Herman79,Yoccoz84}
and~\cite{KatznelsonO89,KhaninS87,SinaiK89,KhaninT09} for finite regularity).
In this section, we present an a-posteriori result that allows constructing
such conjugacy for a given map.  For convenience, we express the result for a
parametric family of circle maps $\al \in A \mapsto \map_\al \in \Difeo{\hat
\rho}$ and, given a target rotation number $\theta$, we obtain a conjugacy
$h_\infty \in \Difeo{\rho_\infty}$ and a parameter $\al_\infty\in A$ (such that
$f_{\al_\infty}$ has rotation number $\theta$).

Given circle maps $f$ and $h$, we measure the error of conjugacy of $f$
to the rigid rotation $R(x)=x+\theta$ through $h$ by introducing the
(periodic) error function
\begin{equation}\label{eq:inva}
e(x):= \map(\conj(x))-\conj(x+\rot)\,. 
\end{equation}

The following a-posteriori result gives
(explicit and quantitative) sufficient conditions to guarantee the existence of
a true conjugacy close to an approximate one.

\begin{theorem}\label{theo:KAM}
Consider an open set $A \subset \RR$, a $C^2$-family $\al \in A
\mapsto \map_\al \in \Difeo{\hat \rho}$, with $\hat \rho > 0$; and a 
rotation number $\theta \in \RR$
fulfilling:
\begin{itemize}
\item [$\GI$]
For every $1 \leq i+j\leq 2$, there exist constants
$c^{i,j}_{x,\al}$ such that
$\norm{\partial_{x,\al}^{i,j} \map_\al}_{A, \hat \rho} \leq c_{x,\al}^{i,j}$.
For convenience we will write $c_x=c_{x,\al}^{1,0}$, $c_{x\al}=c_{x,\al}^{1,1}$, etc.
\item [$\GII$] We have $\theta \in \cD(\gamma,\tau)$.
\end{itemize}
Assume that we
have a pair $(\conj,\al) \in\Difeo{\rho} \times A$, with $\rho > 0$, 
fulfilling:
\begin{itemize}
\item [$\HI$] 
The map $h$ satisfies 
\[
\dist{\conj(\bar \TT_\rho),\partial \TT_{\hat \rho}}>0\,,
\]
and
\begin{equation}\label{eq:normah}
\aver{\conj-\id}=0\,.
\end{equation}
Moreover, there exist constants $\sigh$, and $\sighi$
such that
\[
\norm{\conj'}_\rho < \sigh\,, \qquad
\norm{1/\conj'}_\rho < \sighi\,.
\]
\item [$\HII$] Given 
\begin{equation}\label{eq:aux:fun:b}
b(x) := \frac{\partial_\al \map_\al (h(x))}{\conj'(x+\theta)}
\end{equation}
there exists a constant $\sighb$ such that 
\[
\Abs{1/\aver{b}} < \sighb\,.
\]
\end{itemize}
Then, for any $0<\delta<\rho/2$ and $0<\rho_\infty<\rho-2\delta$,
there exist explicit
constants $\mathfrak{C}_1$, $\mathfrak{C}_2$, and $\mathfrak{C}_3$
(depending on the previously defined constants) such that:
\begin{itemize}
\item [$\TI$]
Existence: If 
\begin{equation}\label{eq:KAM:C1}
\frac{\mathfrak{C}_1\norm{e}_{\rho}}{\gamma^2 \rho^{2\tau}}
< 1\,,
\end{equation}
with $e(x)$ given by \eqref{eq:inva},
then there exists a pair $(\conj_\infty,\al_\infty)
\in \Difeo{\rho_\infty} \times A$ such that
$\map_{\al_\infty}(\conj_\infty(x))=\conj_\infty(x+\rot)$,
that also satisfies $\HI$ and $\HII$.
\item [$\TII$]
Closeness: the pair $(\conj_\infty,\al_\infty)$ is close to the original one:
\begin{equation}\label{eq:KAM:C2}
\norm{\conj_\infty-\conj}_{\rho_\infty} <
\frac{\mathfrak{C}_2 \norm{e}_\rho}{\gamma \rho^\tau}\,,
\qquad
\abs{\al_{\infty}-\al} < \mathfrak{C}_3 \norm{e}_\rho \,.
\end{equation}
\end{itemize}
\end{theorem}

The proof of this result is based in a numerical scheme, proposed in
\cite{LlaveL11}, to compute Arnold ``tongues'' with Diophantine rotation
number. Indeed, the convergence details are adapted from \cite{CanadellH17a},
where the case of torus maps (associated to the inner dynamics of a
toroidal normally hyperbolic manifold) is considered.  We include here a short (but complete)
proof not only for the sake of completeness, but to obtain explicit and sharp
formulae for all the constants and conditions involved. Also, the proof of this
result is the anteroom of the more involved a-posteriori result, discussed in
the next section, which takes into account dependence on parameters.

The construction consists in correcting both the approximate conjugacy
$\conj(x)$ and the parameter $\al$. To this end, we perform an iterative
process introducing the corrected objects $\bar \conj(x)
= \conj(x) + \Delta_\conj(x)$ and $\bar \al= \al + \Delta_\al$. These
corrections are determined by solving approximately the linearized equation
\begin{equation}\label{eq:linear}
\partial_x \map_\al(\conj(x)) \Delta_\conj(x)-\Delta_\conj(x+\rot) + \partial_\al
\map_\al (\conj(x)) \Delta_\al = -e(x)\,,
\end{equation}
where the right\textcolor{blue}{-}hand side is the conjugacy error~\eqref{eq:inva}.
In order to ensure the normalization condition in~\eqref{eq:normah},
we look for a solution such that
\begin{equation}\label{eq:cond:aver}
\aver{\Delta_\conj} =0\,.
\end{equation}

By taking derivatives on both sides of equation~\eqref{eq:inva}, we obtain
\begin{equation}\label{eq:deriv}
\partial_x \map_\al (\conj(x)) \conj'(x) - \conj'(x+\rot) = e'(x) 
\end{equation}
and consider the following transformation
\begin{equation}\label{eq:group}
\Delta_\conj(x)=\conj'(x) \varphi(x)\, .
\end{equation}
Then, introducing~\eqref{eq:deriv} and~\eqref{eq:group}
into~\eqref{eq:linear}, and neglecting the term
$e'(x) \varphi(x)$ (we will see that this term is quadratic in the error) we
obtain that the solution of
equation \eqref{eq:linear} is approximated by the solution of a
cohomological equation for $\varphi$
\begin{equation}\label{eq:coho1}
\varphi(x+\rot) -
\varphi(x)
= \eta(x)\,,
\end{equation}
where the right-hand side is
\begin{equation}\label{eq:eta}
\eta(x):= a(x) + b(x) \Delta_\al\,,
\end{equation}
with
\begin{equation}\label{eq:aux:fun:a}
a(x) := \frac{e(x)}{\conj'(x+\theta)}\,,
\end{equation}
and $b(x)$ is given by~\eqref{eq:aux:fun:b}.

Equation~\eqref{eq:coho1} is solved by expressing $\eta(x)$ and
$\varphi(x)$ in Fourier series.
Then, we obtain a unique zero-average solution, denoted by $\varphi(x) = \cR
\eta(x)$, by taking
\begin{equation}\label{eq:solc}
\Delta_\al = -\frac{\aver{a}}{\aver{b}}\,, \qquad
\hat \varphi_k = \frac{\hat \eta_k}{\ee^{2\pi \ii k \theta}-1}\,,
\qquad k \neq 0\,,
\end{equation}
so that all solutions of~\eqref{eq:coho1} are of the form 
\begin{equation}\label{eq:solc:varphi}
\varphi(x) =
\hat\varphi_0 + \cR \eta(x)\,.
\end{equation}

Finally, the average $\hat \varphi_0 = \aver{\varphi}$ is selected in order
to fulfill condition \eqref{eq:cond:aver}.  Since this condition is equivalent to
$\aver{\conj'\varphi} =0$ we readily obtain
\begin{equation}\label{eq:aver0}
\hat \varphi_0=
-\aver{\conj' \cR \eta}\,.
\end{equation}

The proof of Theorem~\ref{theo:KAM} follows by applying the above
correction iteratively (quasi-Newton method), thus obtaining a sequence of corrected objects.
The norms of each correction and quantitative estimates for the new objects are
controlled 
by applying 
the following result.

\begin{lemma}\label{lem:iter}
Consider an open set $A \subset \RR$, a $C^2$-family $\al \in A \mapsto
\map_\al \in \Difeo{\hat \rho}$, with $\hat\rho>0$; and a rotation number
$\theta \in \RR$ satisfying hypotheses $\GI$ and $\GII$ 
in Theorem~\ref{theo:KAM}.  Assume that we have a pair $(\conj,\al)
\in\Difeo{\rho} \times A$ fulfilling hypotheses $\HI$ and $\HII$ 
in the same theorem.  If the following conditions on the error hold:
\begin{align}
& \frac{\sigh C_1}{\gamma \delta^{\tau}} \norm{e}_\rho < \dist{\conj(\bar \TT_\rho),\partial \TT_{\hat \rho}}\,,
\label{eq:cond:lem:1} \\
& \sighb \sighi \norm{e}_\rho < \dist{\al,\partial A}\,,
\label{eq:cond:lem:2} \\
& \frac{\sigh C_1}{\gamma \delta^{\tau+1}} \norm{e}_\rho < \sigh-\norm{h'}_\rho \,, 
\label{eq:cond:lem:3} \\
& \frac{(\sighi)^2 \sigh C_1}{\gamma \delta^{\tau+1}} \norm{e}_\rho < \sighi-\norm{1/h'}_\rho \,,
\label{eq:cond:lem:4} \\
& \frac{(\sighb)^2 C_2}{\gamma \delta^{\tau+1}} \norm{e}_\rho < \sighb-\abs{1/\aver{b}} \,,
\label{eq:cond:lem:5}
\end{align}
where
\begin{align}
C_1 := {} & (1+\sigh)c_R (1+c_\al \sighb \sighi ) \sighi \,, \label{eq:constC1} \\
C_2 := {} & 
\sigh (\sighi)^2 c_\al C_1 + c_{x\al} \sigh \sighi C_1 \delta +c_{\al \al}
\sighb (\sighi)^2 \gamma \delta^{\tau+1}\,, \label{eq:constC2}
\end{align}
then there exists a new pair
$(\bar \conj,\bar \al) \in \Difeo{\rho-\delta} \times A$,
with $\bar \conj = \conj + \Delta_\conj$
and $\bar \al = \al + \Delta_\al$, satisfying also
hypotheses $\HI$ (in the strip $\TT_{\rho-2\delta}$)
and $\HII$, and the estimates
\begin{equation}\label{eq:defor:al}
\abs{\bar \al - \al} =
\abs{\Delta_\al} 
< \sighb  \sighi \norm{e}_\rho \,
\end{equation}
and
\begin{equation}\label{eq:defor:h}
\norm{\bar \conj - \conj}_{\rho-\delta} =
\norm{\Delta_\conj}_{\rho-\delta} 
< \frac{\sigh C_1 \norm{e}_\rho}{\gamma
\delta^\tau}\,.
\end{equation}
The new conjugacy error, $\bar e(x) = \map_{\bar \al} (\bar
\conj(x))-\bar \conj(x+\rot)$, satisfies
\begin{equation}\label{eq:err}
\norm{\bar e}_{\rho-\delta} < \frac{C_3 \norm{e}_\rho^2}{\gamma^2
\delta^{2\tau}}\,,
\end{equation}
where
\begin{equation}\label{eq:constC3}
C_3 :=
C_1 \gamma \delta^{\tau-1}
+ \frac{1}{2} c_{xx} (\sigh C_1)^2 
+ c_{x \al} \sighb \sigh \sighi C_1 \gamma \delta^\tau 
+ \frac{1}{2} c_{\al \al} (\sighb \sighi)^2 \gamma^2 \delta^{2\tau}
\,.
\end{equation}
\end{lemma}

\begin{proof}
The result follows by controlling the
norms of all the functions involved in the formal scheme described by
equations~\eqref{eq:linear} to~\eqref{eq:aver0}.
Using Cauchy estimates, we have $\norm{e'}_{\rho-\delta} \leq \delta^{-1}
\norm{e}_\rho$. Using hypothesis $\HI$ we directly control the
functions in~\eqref{eq:aux:fun:b} and~\eqref{eq:aux:fun:a}
\begin{equation}\label{eq:control:a:b}
\norm{b}_\rho \leq \norm{1/h'}_\rho \norm{\partial_\al \map_\al}_{A, \hat \rho} < \sighi c_\al\,,
\qquad
\norm{a}_\rho \leq \norm{1/h'}_\rho \norm{e}_\rho < \sighi \norm{e}_\rho\,.
\end{equation}
Then, using also $\HII$, the expression for
$\Delta_\al$ in \eqref{eq:solc} is controlled as
\[
\abs{\Delta_\al} \leq \abs{1/\aver{b}} \abs{\aver{a}} < \sighb \sighi \norm{e}_\rho\,,
\]
thus obtaining~\eqref{eq:defor:al}.
Hence
we control the function $\eta(x)$ in~\eqref{eq:eta} as
\[
\norm{\eta}_\rho 
\leq \norm{a}_\rho + \norm{b}_\rho\abs{\Delta_\al}
< (1+c_\al \sighb \sighi) \sighi \norm{e}_\rho \,.
\]
By decomposing $\varphi(x)=\hat \varphi_0+\cR \eta(x)$ and invoking
Lemma~\ref{lem:Russ}, with hypothesis $\GII$,
we obtain
\begin{equation}\label{eq:eta:c}
\norm{\cR \eta}_{\rho-\delta} \leq \frac{c_R}{\gamma \delta^\tau}
\norm{\eta}_\rho < \frac{c_R (1+c_\al \sighb \sighi) \sighi}{\gamma \delta^\tau} \norm{e}_\rho\, ,
\end{equation}
and we control the average $\hat \varphi_0$ using the expression \eqref{eq:aver0},
hypothesis $\HI$, and~\eqref{eq:eta:c}:
\begin{equation}\label{eq:eta:0}
\abs{\hat \varphi_0} \leq \norm{\conj'}_\rho \norm{\cR \eta}_{\rho-\delta} <
\sigh\frac{c_R (1+c_\al \sighb \sighi) \sighi}{\gamma
\delta^\tau} \norm{e}_\rho\,.
\end{equation}
By grouping expressions~\eqref{eq:eta:c} and~\eqref{eq:eta:0} we obtain
\begin{equation}\label{eq:varphi}
\norm{\varphi}_{\rho-\delta} \leq |\hat \varphi_0|+\norm{\cR\eta}_{\rho-\delta}< \frac{C_1}{\gamma \delta^\tau}
\norm{e}_{\rho}\, ,
\end{equation}
where $C_1$ is given in~\eqref{eq:constC1}.
Finally,
using~\eqref{eq:group} and $\HI$,
we obtain~\eqref{eq:defor:h}.

Now we control of the distance of the corrected objects to the boundaries of the
domains, i.e. we ensure that $\dist{\conj(\bar \TT_\rho),\partial \TT_{\hat
\rho}}>0$ and $\dist{\bar \al,\partial A} >0$. Using the assumption in \eqref{eq:cond:lem:1}
we have
\begin{align*}
\dist{\bar \conj(\bar \TT_{\rho-\delta}),\partial \TT_{\hat \rho}} \geq {} &
\dist{\conj(\bar \TT_\rho),\partial \TT_{\hat \rho}} -
\norm{\Delta_\conj}_{\rho-\delta} \\
> {} & \dist{\conj(\bar \TT_\rho),\partial \TT_{\hat \rho}} - \frac{\sigh
C_1}{\gamma \delta^\tau} \norm{e}_\rho>0 \,,
\end{align*}
and using the assumption in \eqref{eq:cond:lem:2}
we have
\[
\dist{\bar \al,\partial A} \geq 
\dist{\al,\partial A} - |\Delta_\al| > 
\dist{\al,\partial A} - \sighb \sighi\norm{e}_\rho >0 \,.
\]

Next we check that the new approximate conjugacy $\bar \conj(x)=\conj(x)+\Delta_\conj(x)$
satisfies $\HI$ in the strip $\TT_{\rho-2\delta}$:
\[
\norm{\bar \conj'}_{\rho-2\delta} < \sigh\,,
\qquad
\norm{1/\bar \conj'}_{\rho-2\delta} < \sighi\,,
\qquad
\abs{1/\aver{\bar b}} < \sighb\,.
\]

The first inequality in $\HI$ follows directly
using Cauchy estimates in~\eqref{eq:defor:h} and the assumption in~\eqref{eq:cond:lem:3}
\textcolor{blue}{:}
\begin{equation}\label{eq:mycond}
\norm{\bar \conj'}_{\rho-2\delta} \leq \norm{\conj'}_{\rho-2\delta} +
\norm{\Delta_\conj'}_{\rho-2\delta} < \norm{\conj'}_{\rho} + \frac{\sigh
C_1}{\gamma\delta^{\tau+1}} \norm{e}_\rho < \sigh\,.
\end{equation}

The second inequality follows using Neumann series: in general,
if $m \in \CC$ satisfies $|1/m|<\sigma$ and $\bar m \in \CC$ satisfies
\begin{equation}\label{eq:Neumann:hyp}
\frac{\sigma^2 \abs{\bar m-m}}{\sigma-\abs{1/m}} < 1\,,
\end{equation}
then we have
\begin{equation}\label{eq:Neumann:tesis}
\abs{1/\bar m} < \sigma\,,
\qquad
\abs{1/\bar m-1/m} < \sigma^2 \abs{\bar m-m}\,.
\end{equation}

Since hypothesis~\eqref{eq:cond:lem:4} implies
\begin{equation}\label{eq:checkH2}
\frac{(\sighi)^2 \norm{\bar \conj'-\conj'}_{\rho-2\delta}}{\sighi - \norm{1/\conj'}_\rho} \leq
\frac{(\sighi)^2 \sigh C_1}{\sighi - \norm{1/\conj'}_\rho} \frac{\norm{e}_\rho}{\gamma \delta^{\tau+1}} < 1\,,
\end{equation}
from~\eqref{eq:Neumann:hyp} and~\eqref{eq:Neumann:tesis} we obtain
\begin{equation}\label{eq:diffDh}
\norm{1/\bar \conj'-1/\conj'}_{\rho-2\delta} < (\sighi)^2 \norm{\bar \conj'-\conj'}_{\rho-2\delta} < \frac{(\sighi)^2 \sigh C_1}{\gamma \delta^{\tau+1}} \norm{e}_\rho\,,
\end{equation}
and also that the control $\norm{1/\bar \conj'}_{\rho-2\delta} < \sighi$ is preserved.

The control of $\abs{1/\aver{\bar b}} < \sighb$ in $\HII$ is similar. To this end, we first
write
\begin{align*}
\bar b(x) - b(x) = {} & 
\frac{\partial_\al \map_{\bar \al} (\bar \conj(x))}{\bar \conj'(x+\theta)}
- 
\frac{\partial_\al \map_{\bar \al} (\conj(x))}{\bar \conj'(x+\theta)} 
+
\frac{\partial_\al \map_{\bar \al} (\conj(x))}{\bar\conj'(x+\theta)}
- 
\frac{\partial_\al \map_\al (\conj(x))}{\bar\conj'(x+\theta)} \\
& +
\frac{\partial_\al \map_{\al} (\conj(x))}{\bar \conj'(x+\theta)}
- 
\frac{\partial_\al \map_{\al} (\conj(x))}{\conj'(x+\theta)} \,,
\end{align*}
so, using~\eqref{eq:defor:al}, \eqref{eq:defor:h} and~\eqref{eq:diffDh}, we get the estimate
\[
\left| \aver{\bar b} -\aver{b} \right| \leq 
c_{x\al} \sighi \norm{\Delta_\conj}_{\rho-\delta}+
c_{\al \al} \sighi \abs{\Delta_\al}
+c_\al \norm{1/\bar \conj'-1/\conj'}_{\rho-2\delta} < \frac{C_2}{\gamma \delta^{\tau+1}} \norm{e}_\rho \,,
\]
where $C_2$ is given by~\eqref{eq:constC2}.
Then, we repeat the computation in~\eqref{eq:checkH2}
using the assumption in \eqref{eq:cond:lem:5}. This provides the 
condition $\abs{1/\aver{\bar b}}<\sighb$ in $\HII$.

Finally, the new error of invariance is given by
\[
\bar e(x) = \map_{\bar \al}(\bar \conj(x))-\bar \conj(x+\rot) = e'(x) \varphi(x) +
\Delta^2\map(x) \,,
\]
where
\begin{align}
\Delta^2 \map (x) = {} &
f_{\bar\al}(\bar h(x))-f_\al(h(x))-\pd_x f_\al(h(x)) \Delta_h(x)-\pd_\al f_\al(h(x)) \Delta_\al \nonumber \\
= {} &
\int_0^1 (1-t) 
\left(
 G_{xx} (x)\Delta_\conj(x)^2 
+ 2 G_{x \al}(x)  \Delta_\conj(x) \Delta_\al
+G_{\al\al} (x)\Delta_\al^2
\right)
dt\,,
\label{eq:D2F}
\end{align}
and we use the notation $G_{xx}(x):=\partial_{xx} f_{\al+t \Delta_\al}(\conj(x)+t \Delta_\conj(x))$
and similar for $G_{x \al}(x)$ and $G_{\al \al}(x)$.
Using $\GI$
and the estimates~\eqref{eq:defor:al} and~\eqref{eq:defor:h} for the corrections, we obtain
the bound in~\eqref{eq:err}.
\end{proof}

\begin{proof}[Proof of Theorem~\ref{theo:KAM}]
To initialize the iterative method,
we introduce the notation
$\conj_0=\conj$, $\al_0=\al$, and $e_0=e$. Notice that Lemma~\ref{lem:iter}
provides
control of the analytic domains after each iteration. 
At the $s$-th iteration, we denote
$\rho_s$ the strip of analyticity (with $\rho_0=\rho$) and
$\delta_s$ the loss of strip produced at this step (with $\delta_0=\delta$). Then, we take
\begin{equation}\label{eq:a1}
\rho_s = \rho_{s-1}-2\delta_{s-1}\,,\qquad
\delta_s = \frac{\delta_{s-1}}{a_1}\,,\qquad
a_1 := \frac{\rho_0- \rho_\infty}{\rho_0-2\delta_0-\rho_\infty}>1\,,
\end{equation}
where $\rho_\infty$ is the final strip.  For convenience, we introduce
the auxiliary constants
\begin{equation}\label{eq:a2a3}
a_2 = \frac{\rho_0}{\rho_\infty} > 1\,, \qquad
a_3 = \frac{\rho_0}{\delta_0}> 2\,,
\end{equation}
and observe that the following relation involving $a_1$,
$a_2$, and $a_3$ holds
\[
a_3 = 2 \frac{a_1}{a_1-1} \frac{a_2}{a_2-1}\,.
\]
In accordance to the previous notation,
we denote by $\conj_s$, $\al_s$, and
$e_s$
the objects at the
$s$-step of the quasi-Newton method.

\bigskip
\emph{Existence:} We proceed by induction, assuming that 
we have successfully applied  $s$ times
Lemma~\ref{lem:iter}.
At this point,
we use~\eqref{eq:err} to
control the error of the last conjugacy in terms of the initial one:
\begin{equation}\label{eq:conv:err}
\begin{split}
\norm{e_s}_{\rho_s} < {} & 
\frac{C_3}{\gamma^2 \delta_{s-1}^{2\tau}} \norm{e_{s-1}}_{\rho_{s-1}}^2 =
\frac{C_3 a_1^{2\tau(s-1)}}{\gamma^2 \delta_{0}^{2\tau}} \norm{e_{s-1}}_{\rho_{s-1}}^2 \\
< {} & \left(
\frac{C_3 a_1^{2\tau }}{\gamma^2 \delta_{0}^{2\tau}}
\norm{e_0}_{\rho_0} 
\right)^{2^s-1}
a_1 ^{-2\tau s} \norm{e_0}_{\rho_0} 
\,,
\end{split}
\end{equation}
where we used that $1+2+\ldots+2^{s-1}=2^s-1$ and
$1(s-1)+2(s-2)+\ldots+2^{s-2}1=2^s-s-1$. 
The above computation motivates the
condition
\begin{equation}\label{eq:cond1}
\kappa:=
\frac{C_3 a_1^{2\tau }}{\gamma^2 \delta_{0}^{2\tau}}
\norm{e_0}_{\rho_0} < 
1 \,,
\end{equation}
included in~\eqref{eq:KAM:C1}. 
Under this condition, the sum
\begin{equation}\label{eq:mysum}
\Sigma_{\kappa,\lambda}:=
\sum_{j=0}^\infty 
\kappa^{2^j-1} a_1^{-\lambda j}
\end{equation}
is convergent for any $\lambda \in \RR$.

Now, using expression~\eqref{eq:conv:err}, we check the inequalities in
Lemma~\ref{lem:iter} (that is \eqref{eq:cond:lem:1}, \eqref{eq:cond:lem:2},
\eqref{eq:cond:lem:3}, \eqref{eq:cond:lem:4}, and~\eqref{eq:cond:lem:5}), so
that we can perform the step $s+1$. The required sufficient condition will be
also included in~\eqref{eq:KAM:C1}.  To simplify the computations, we 
consider that the constants $C_1$, $C_2$, and $C_3$ are evaluated at the worst
value of $\delta_s$ (that is $\delta_0$) so that they can be taken 
to be equal at all steps. For example,
the inequality~\eqref{eq:cond:lem:1} is obtained as follows:
\begin{align}
\dist{\conj_s(\bar \TT_{\rho_s}),\partial \TT_{\hat \rho}} - \frac{\sigh
C_1 \norm{e_s}_{\rho_s} }{\gamma \delta_s^{\tau}} 
> {} & \dist{\conj_0(\bar \TT_{\rho_0}),\partial \TT_{\hat \rho}} - \sum_{j=0}^\infty \frac{\sigh
C_1 \norm{e_j}_{\rho_j}}{\gamma \delta_j^{\tau}} \nonumber \\
> {} &\dist{\conj_0(\bar \TT_{\rho_0}),\partial \TT_{\hat \rho}} - \frac{\sigh
C_1 \Sigma_{\kappa,\tau} \norm{e_0}_{\rho_0}}{\gamma \delta_0^{\tau}}
> 0\,, \label{eq:KAM:C1:cond1}
\end{align}
where we used $\conj_s=\conj_{s-1}+\Delta_{\conj_{s-1}}$, \eqref{eq:conv:err} and \eqref{eq:mysum}.
The last inequality in~\eqref{eq:KAM:C1:cond1} is included in~\eqref{eq:KAM:C1}.
The control of~\eqref{eq:cond:lem:2} is completely analogous:
\begin{equation}\label{eq:KAM:C1:cond2}
\dist{\al_s,\partial A} - \sighi \sigh \norm{e_s}_{\rho_s}
> \dist{\al_0,\partial A} - \sighi \sigh \Sigma_{\kappa,2\tau} \norm{e_0}_{\rho_0}
> 0\,, 
\end{equation}
and the last inequality in~\eqref{eq:KAM:C1:cond2} is included in~\eqref{eq:KAM:C1}.
The condition in~\eqref{eq:cond:lem:3} is obtained using
$\conj_s=\conj_{s-1}+\Delta_{\conj_{s-1}}$ and \eqref{eq:conv:err}:
\begin{align}
\norm{\conj_s'}_{\rho_s} + \frac{\sigh C_1}{\gamma \delta_s^{\tau+1}}
\norm{e_s}_{\rho_s} & < \norm{\conj_0'}_{\rho_0} + \sum_{j=0}^s
\frac{\sigh
C_1}{\gamma \delta_j^{\tau+1}} \norm{e_j}_{\rho_j} \nonumber \\
& < \norm{\conj_0'}_{\rho_0} + \frac{\sigh C_1 \Sigma_{\kappa,\tau-1}}{\gamma \delta_0^{\tau+1}}
\norm{e_0}_{\rho_0} < \sigh\,, \label{eq:KAM:C1:cond3}
\end{align}
and the last inequality in~\eqref{eq:KAM:C1:cond3} is included in~\eqref{eq:KAM:C1}.
Analogous computations allow us to guarantee the conditions in~\eqref{eq:cond:lem:4}
and~\eqref{eq:cond:lem:5} for
$\norm{1/\conj_s'}_{\rho_s}$ and $\abs{1/\aver{b_s}}$,
respectively. To this end, we also include in~\eqref{eq:KAM:C1} the inequalities
\begin{align}
&\norm{1/\conj_0'}_{\rho_0} + \frac{(\sighi)^2\sigh C_1 \Sigma_{\kappa,\tau-1}}{\gamma \delta_0^{\tau+1}}
\norm{e_0}_{\rho_0} < \sighi \,, \label{eq:KAM:C1:cond4} \\
&\abs{1/\aver{b_0}} + \frac{(\sighb)^2C_2 \Sigma_{\kappa,\tau-1}}{\gamma \delta_0^{\tau+1}}
\norm{e_0}_{\rho_0} < \sighb \,. \label{eq:KAM:C1:cond5}
\end{align}

Putting together the assumptions in~\eqref{eq:cond1},
\eqref{eq:KAM:C1:cond1},
\eqref{eq:KAM:C1:cond2},
\eqref{eq:KAM:C1:cond3},
\eqref{eq:KAM:C1:cond4},
\eqref{eq:KAM:C1:cond5},
and recalling that
$\rho/\delta=\rho_0/\delta_0=a_3$, $h=h_0$, $\al=\al_0$, we end up with
\begin{equation}\label{eq:frakC1}
\mathfrak{C}_1 := 
\left\{
\begin{array}{ll}
\max \left\{
(a_1 a_3)^{2 \tau}
C_3, (a_3)^{\tau+1} C_4 \gamma \rho^{\tau-1}
\right\}  & \mbox{if $\kappa<1$} \,, \\
\infty & \mbox{otherwise} \,,
\end{array}
\right.
\end{equation}
where $\kappa$ is given by~\eqref{eq:cond1} and
\begin{align}
C_4 := \max \Bigg\{ &
\frac{\sigh C_1 \delta \Sigma_{\kappa,\tau}}{\dist{\conj(\bar \TT_{\rho}),\partial \TT_{\hat \rho}}}\,,
\frac{\sighi \sigh \Sigma_{\kappa,2\tau} \gamma \delta^{\tau+1}}{\dist{\al,\partial A}}\,, \label{eq:const4} \\
& \frac{\sigh C_1 \Sigma_{\kappa,\tau-1}}{\sigh-\norm{\conj'}_\rho}\,,
\frac{(\sighi)^2 \sigh C_1 \Sigma_{\kappa,\tau-1}}{\sighi-\norm{1/\conj'}_\rho}\,,
\frac{(\sighb)^2 C_2 \Sigma_{\kappa,\tau-1}}{\sighb-\abs{1/\aver{b}}}
\Bigg\} \,.\nonumber
\end{align}

As a consequence of the above computations, we can apply
Lemma~\ref{lem:iter} again. 
By induction, we obtain
a convergent sequence
$\norm{e_s}_{\rho_s} \rightarrow 0$ when $s \rightarrow \infty$.
Conclusively, the
iterative scheme converges to a true conjugacy $\conj_{\infty} \in
\Difeo{\rho_\infty}$ for the map
$\map_{\infty}=\map_{\al_{\infty}}$.

\bigskip
\emph{Closeness:} The above computations also prove that the
conjugacy $\conj_{\infty}$ and the parameter $\al_{\infty}$
are close to the initial objects:
\begin{align*}
& \norm{\conj_\infty-\conj}_{\rho_\infty} < \sum_{j=0}^\infty
\norm{\Delta_{\conj_j}}_{\rho_j} <
\sum_{j=0}^\infty \frac{\sigh C_1 \norm{e_j}_{\rho_j}}{\gamma \delta_j^\tau}
< \frac{\sigh C_1 \Sigma_{\kappa,\tau}}{\gamma \delta_0^\tau} \norm{e_0}_{\rho_0}\,,\\
& \abs{\al_\infty-\al} < \sighb \sighi \sum_{j=0}^\infty \norm{e_j}_{\rho_j} <
\sighb \sighi \Sigma_{\kappa,2\tau} \norm{e_0}_{\rho_0}\,,
\end{align*}
and we finally obtain
\begin{equation}\label{eq:frakC2C3}
\mathfrak{C}_2 = a_3^\tau \sigh C_1 \Sigma_{\kappa,\tau} \,,
\qquad
\mathfrak{C}_3 = \sighb \sighi \Sigma_{\kappa,2\tau}\,. \qedhere
\end{equation}
\end{proof}

\begin{remark}
Here we summarize how to compute the constants
$\mathfrak{C}_1$,
$\mathfrak{C}_2$,
$\mathfrak{C}_3$ in Theorem~\ref{theo:KAM}. Given fixed values
of the parameters
$\rho,\delta,\rho_\infty,\hat\rho$
and the distances
$\dist{\conj(\bar \TT_\rho),\partial \TT_{\hat \rho}}$
and
$\dist{\al,\partial A}$; the constants
$c_x$, $c_\al$, $c_{xx}$, $c_{x\al}$, $c_{\al\al}$ in hypothesis $\GI$;
the constants $\gamma$ and $\tau$ in hypothesis $\GII$;
the constants $\sigma_1$, $\sigma_2$ in hypothesis $\HI$;
and the constant $\sigma_b$ in hypothesis $\HII$, these are computed in the following order:
\begin{itemize}[leftmargin=5mm]
\item $a_1$, $a_2$, $a_3$ using~\eqref{eq:a1} and~\eqref{eq:a2a3}.
\item $C_1$, $C_2$, $C_3$ using~\eqref{eq:constC1}, \eqref{eq:constC2}
and~\eqref{eq:constC3}.
\item $\kappa$ using~\eqref{eq:cond1} and check that $\kappa<1$ (abort the process otherwise).
\item $\Sigma_{\kappa,\tau}$,
$\Sigma_{\kappa,2\tau}$,
$\Sigma_{\kappa,\tau-1}$ using~\eqref{eq:mysum}.
\item $C_4$ using~\eqref{eq:const4}.
\item $\mathfrak{C}_1$,
$\mathfrak{C}_2$,
$\mathfrak{C}_3$ using~\eqref{eq:frakC1} and~\eqref{eq:frakC2C3}.
\end{itemize}
\end{remark}

\section{An a-posteriori theorem for the measure of conjugations}\label{ssec:conjL}

In this section we present an a-posteriori result that extends
Theorem~\ref{theo:KAM} considering dependence on parameters.  The statement is
a detailed version of the theorem that was informally exposed in the
introduction of the paper.  Our aim is to use a global approximation of the
conjugacies to rotation, within a range of rotation numbers, in order to obtain
a lower bound of the measure of the set of parameters for which a true
conjugation exists.

\begin{theorem}\label{theo:KAM:L}
Consider an open set $A \subset \RR$,
a $C^3$-family $\al \in A \mapsto \map_\al \in \Difeo{\hat \rho}$,
with $\hat \rho>0$,
and a set $\Theta \subset \RR$ 
fulfilling:
\begin{itemize}
\item [$\GI$]
For every $1 \leq i+j\leq 3$, there exist constants
$c^{i,j}_{x,\al}$ such that
$\norm{\partial_{x,\al}^{i,j} \map_\al}_{A, \hat \rho} \leq c_{x,\al}^{i,j}$.
For convenience we will write $c_x=c_{x,\al}^{1,0}$, $c_{x\al}=c_{x,\al}^{1,1}$, $c_{\al\al\al}=c_{x,\al}^{0,3}$, etc.
\item [$\GII$] 
We have $\Theta \subset \cD(\gamma,\tau)$.
\end{itemize}
Assume that we have a family of pairs
\[
\theta \in \Theta \longmapsto 
(h_\theta,\al(\theta))
\in \Difeo{\rho} \times A\,,
\]
with $\rho>0$,
such that the following quantitative estimates are satisfied:
\begin{itemize}
\item [$\HI$]
For every $\theta \in \Theta$ the map
$h_\theta$ satisfies 
\[
\dist{\conj_\theta(\bar \TT_\rho),\partial \TT_{\hat \rho}}>0\,,
\]
and
\begin{equation}\label{eq:normah:L}
\aver{\conj_\theta-\id}=0\,.
\end{equation}
Moreover, there exist constants $\sigh$, $\sighi$ and $\sighxx$
such that
\[
\norm{\pd_x \conj}_{\Theta,\rho} < \sigh\,, \qquad
\norm{1/\pd_x \conj}_{\Theta,\rho} < \sighi\,, \qquad
\norm{\pd_{xx} \conj}_{\Theta,\rho} < \sighxx\,.
\]
\item [$\HII$] The average of the family of functions
\begin{equation}\label{eq:aux:fun:b:L}
b_\theta(x) := \frac{\partial_\al \map_{\al(\theta)} (h_\theta(x))}{\pd_x \conj_\theta(x+\theta)}
\end{equation}
is different from zero for every $\theta \in \Theta$. Moreover, there exist a constant $\sighb$ such that 
\[
\norm{1/\aver{b}}_{\Theta} < \sighb\,.
\]
\item [$\HIII$] There exist constants $\betah$,
$\betaDh$,
and $\betaa$ such that
\[
\Lip{\Theta,\rho}(\conj-\id) < \betah\,, \qquad
\Lip{\Theta,\rho}(\partial_x \conj) < \betaDh\,,
\qquad
\Lip{\Theta}(\al) < \betaa\,.
\]
\end{itemize}
Then, for any $0<\delta<\rho/2$ and $0<\rho_\infty<\rho-2\delta$,
there exist constants $\mathfrak{C}_1^\sLip$ and $\mathfrak{C}_2^\sLip$
such that:
\begin{itemize}
\item [$\TI$]
Existence of conjugations: If 
\begin{equation}\label{eq:KAM:C1:L}
\frac{\mathfrak{C}_1^\sLip
\max \left\{ \norm{e}_{\Theta,\rho} \,,\,\gamma \delta^{\tau+1}\Lip{\Theta,\rho}(e)\right\}}{\gamma^2 \rho^{2\tau+2}}
< 1\,,
\end{equation}
where
\[
e_\theta(x):= \map_{\al(\rot)}(\conj_\theta(x))-\conj_\theta(x+\rot)
\]
is the associated family of error functions,
then there exists a family of pairs
$\theta \in \Theta \mapsto (h_{\theta,\infty},\al_\infty(\theta)) \in \Difeo{\rho_\infty} \times A$
such that
\[
\map_{\al_\infty(\theta)}(\conj_{\theta,\infty}(x))=\conj_{\theta,\infty}(x+\rot)\,,\qquad
\forall \theta \in \Theta
\]
and also satisfying $\HI$, $\HII$, and $\HIII$.
\item [$\TII$] Closeness: the family $\theta \in \Theta \mapsto
(h_{\theta,\infty}, \al_\infty (\theta))$ is close
to the original one:
\begin{equation}\label{eq:KAM:C2:L}
\norm{\conj_\infty-\conj}_{\Theta,\rho_\infty} <
\frac{\mathfrak{C}_2 \norm{e}_{\Theta,\rho}}{\gamma \rho^\tau}\,,
\qquad
\norm{\al_{\infty}-\al}_{\Theta} < \mathfrak{C}_3 \norm{e}_{\Theta,\rho} \,,
\end{equation}
where the constants 
$\mathfrak{C}_2$
and
$\mathfrak{C}_3$
are computed in the same way as the analogous
constants in Theorem~\ref{theo:KAM}, i.e., are given by~\eqref{eq:frakC2C3}.
\item [$\TIV$] Measure of rotations: the Lebesgue measure of conjucacies 
to rigid rotation
in the space of parameters is bounded from below as follows
\[
\meas (\al_\infty(\Theta)) >
\left[
\lip{\Theta}(\al)
- 
\frac{\mathfrak{C}_2^\sLip
\max \left\{ \norm{e}_{\Theta,\rho} \,,\,\gamma \delta^{\tau+1}\Lip{\Theta,\rho}(e)\right\}}{\gamma \rho^{\tau+1}}
\right]
\meas(\Theta)\,.
\]
\end{itemize}
\end{theorem}

\begin{remark}
Notice that hypothesis $\HII$ plays the role of a \emph{twist condition},
guaranteeing
that each rotation number in the domain appears exactly one.
If this condition is violated in a particular example, then the interval $A$ should
be divided into subintervals and apply the theorem in each of them.
\end{remark}

\begin{remark}\label{rem:common}
Notice that, replacing the norms $\norm{\cdot}_\rho$ and $\abs{\cdot}$
by $\norm{\cdot}_{\Theta,\rho}$ and $\norm{\cdot}_\Theta$ respectively,
we have that part of the hypotheses of Theorem~\ref{theo:KAM:L}
are in common with those of Theorem~\ref{theo:KAM}. Taking this
into account, we will omit the details associated to the control
of the norms $\norm{\cdot}_\rho$ and $\abs{\cdot}$, since we can
mimic these estimates from those obtained in Section~\ref{ssec:conj}.
\end{remark}

The proof of this result follows by adapting the construction presented in
Section~\ref{ssec:conj}, but controlling the Lipschitz constants of the objects
involved in the iterative scheme. 
Notice that our main interest is to show that
$\al_\infty$ is Lipschitz from below and to obtain sharp lower bounds.
To this end, using that
$\al_s=\al_0+\al_s-\al_0$, we have that
\[
\lip{\Theta}(\al_s) > \lip{\Theta}(\al_0) -
\Lip{\Theta}(\al_{s}-\al_{0})\,,
\]
so the effort is focused in controlling the Lipschitz constants from above.

The following result provides quantitative estimates for the
norm of the corrected objects.

\begin{lemma}\label{lem:iter:L}
Consider an open set $A \subset \RR$,
a $C^3$-family $\al \in A \mapsto \map_\al \in \Difeo{\hat \rho}$,
with $\hat\rho>0$, and a set $\Theta \subset \RR$
satisfying the hypotheses $\GI$ and $\GII$ of Theorem \ref{theo:KAM:L}.
Assume that we have a family of pairs
\[
\theta \in \Theta \longmapsto 
(h_\theta,\al(\theta))
\in \Difeo{\rho} \times A
\]
fulfilling hypotheses $\HI$, $\HII$, and $\HIII$ of the same theorem.  
Assume that, adapting the expressions for the norms
$\norm{\cdot}_{\Theta,\rho}$ and $\norm{\cdot}_{\Theta}$, the family of errors
$\theta \in \Theta \mapsto e_\theta$ satisfies
conditions like~\eqref{eq:cond:lem:1},
\eqref{eq:cond:lem:2}, \eqref{eq:cond:lem:3}, \eqref{eq:cond:lem:4},
\eqref{eq:cond:lem:5}, and the additional conditions:
\begin{align}
& 
\frac{\CLiph}{\gamma^2 \delta^{2\tau+1}} \norm{e}_{\Theta,\rho} +
\frac{\sigh C_1}{\gamma \delta^\tau} \Lip{\Theta,\rho}(e) < \betah - \Lip{\Theta,\rho}(h-\id)\,,
\label{eq:cond:lem:1:L} \\
&
\frac{\CLiph}{\gamma^2 \delta^{2\tau+2}} \norm{e}_{\Theta,\rho} +
\frac{\sigh C_1}{\gamma \delta^{\tau+1}} \Lip{\Theta,\rho}(e) < \betaDh - \Lip{\Theta,\rho}(\pd_x h)\,,
\label{eq:cond:lem:2:L} \\
& 
\CLipO \norm{e}_{\Theta,\rho}+\sighb \sighi \Lip{\Theta,\rho}(e) < \betaa - \Lip{\Theta}(\al)\,,
\label{eq:cond:lem:3:L} \\
& \frac{2\sigh C_1}{\gamma \delta^{\tau+2}} \norm{e}_{\Theta,\rho} < \sighxx-
\norm{\partial_{xx}h}_{\Theta,\rho}\,,
\label{eq:cond:lem:4:L} 
\end{align}
where constant $C_1$ is computed as in~\eqref{eq:constC1} and we introduce the new constants
\begin{align}
& \CLipO :=
\sighb (\sighi)^2 (\betaDh+\sighxx)(1+c_\al \sighb\sighi )
+ (\sighi)^2 (\sighb)^2 [c_{\al \al} \betaa + c_{x \al}
\betah]
\label{eq:C0L}\\
& \CLipI := 
\sighi[(c_{\al \al} \betaa+c_{x\al} \betah) \sighb \sighi + c_\al C_0^\sLip
+ (1+c_\al \sighb \sighi) \sighi (\betaDh+\sighxx)]
\,, \label{eq:C1L} \\
& \CLipII := c_R \CLipI \gamma \delta^{\tau+1}+\hat c_R(1+c_\al\sighb \sighi)\sighi \,, \label{eq:CLipII} \\
& \CLipIII := 
\CLipII (\sigh+1) + \betaDh c_R(1+c_\al \sighb \sighi ) \sighi\gamma \delta^{\tau+1} \,,
\label{eq:C2L} \\
& \CLiph :=
\betaDh C_1 \gamma \delta^{\tau+1}  + \sigh \CLipIII \,.
\label{eq:ChL} 
\end{align}
Then, there is a new family $\theta \in \Theta \mapsto
\bar \conj_\theta \in \Difeo{\rho-\delta}$,
with $\bar \conj_\theta = \conj_\theta + \Delta_{\conj_\theta}$,
and a new function
$\theta \in \Theta \mapsto \bar \al \in A$, with
$\bar \al(\theta) = \al(\theta) + \Delta_{\al}(\theta)$, satisfying
also hypotheses $\HI$, $\HII$, and $\HIII$ of Theorem~\ref{theo:KAM:L}
in the strip $\TT_{\rho-2\delta}$; and the estimates
\begin{equation}\label{eq:defor:al:L}
\norm{\bar \al - \al}_\Theta =
\norm{\Delta_\al}_\Theta
< \sighb \sighi \norm{e}_{\Theta,\rho} \,,
\end{equation}
\begin{equation}\label{eq:defor:h:L}
\norm{\bar \conj - \conj}_{\Theta,\rho-\delta} =
\norm{\Delta_\conj}_{\Theta,\rho-\delta} 
< \frac{\sigh C_1 \norm{e}_{\Theta,\rho}}{\gamma
\delta^\tau}\,,
\end{equation}
\begin{equation}\label{eq:defor:al:L2}
\Lip{\Theta}(\Delta_\al)
< 
C_0^\sLip \norm{e}_{\Theta,\rho}
+
\sighb \sighi \Lip{\Theta,\rho}(e) 
\,,
\end{equation}
and
\begin{equation}\label{eq:defor:h:L2}
\Lip{\Theta,\rho-\delta}(\Delta_\conj)
< 
\frac{\CLiph}{\gamma^2 \delta^{2\tau+1}}
\norm{e}_{\Theta,\rho}+
\frac{\sigh C_1 }{\gamma
\delta^\tau} \Lip{\Theta,\rho}(e)\,.
\end{equation}
The new error $\bar e_\theta(x) = \map_{\bar \al(\theta)} (\bar
\conj_\te(x))-\bar \conj_\te(x+\rot)$, satisfies
\begin{equation}
\label{eq:errL:L} 
\norm{\bar e}_{\Theta,\rho-\delta} < \frac{C_3 \norm{e}_{\Theta,\rho}^2}{\gamma^2
\delta^{2\tau}}\,,
\end{equation}
and
\begin{equation}
\label{eq:errL:L2} 
\Lip{\Theta,\rho-\delta}(\bar e) < 
\frac{\CLipIV \norm{e}_{\Theta,\rho}^2}{
\gamma^3
\delta^{3\tau+1}}+
\frac{2 C_3 \Lip{\Theta,\rho}(e) \norm{e}_{\Theta,\rho}}{\gamma^2
\delta^{2\tau}}\,,
\end{equation}
where $C_3$ is computed as in~\eqref{eq:constC3} and
\begin{align}
& \CLipIV := 
\CLipIII \gamma \delta^{\tau-1} +\tfrac{1}{2} (c_{xx\al}\betaa+c_{xxx}\betah)(\sigh C_1)^2 
\gamma \delta^{\tau+1}
\label{eq:C3L} \\
& \quad + (c_{x\al\al} \betaa+c_{xx\al} \betah) 
\sigh C_1 \sighb \sighi \gamma^2 \delta^{2\tau+1} \nonumber \\
& \quad + 
\tfrac{1}{2}(c_{\al\al\al}\betaa+c_{x\al\al} \betah) (\sighb \sighi)^2
\gamma^3 \delta^{3\tau+1}
\nonumber \\
& \quad + \CLiph (c_{xx}
\sigh C_1 + c_{x\al} \sighb \sighi\gamma \delta^\tau) \nonumber \\
& \quad + C_0^\sLip \gamma^2 \delta^{2\tau+1} (c_{x\al} \sigh C_1 + c_{\al \al} \sighb \sighi
\gamma\delta^\tau) \,, \nonumber
\end{align}
\end{lemma}

\begin{proof}
Following the proof of Lemma~\ref{lem:iter} we consider the objects
\begin{equation}\label{eq:eta:L}
\eta_\te(x):= a_\te(x) + b_\te(x) \Delta_\al(\theta)\,,
\qquad
\varphi_\te(x)=\hat \varphi_0(\theta)+\cR_\te \eta_\te(x)\,,
\end{equation}
where
\begin{equation}\label{eq:Deltaa:L}
a_\te(x)=\frac{e_\te(x)}{\partial_x h_\te(x+\te)}\,,
\qquad
\Delta_\al(\theta) = -\frac{\aver{a_\te}}{\aver{b_\te}}\,,
\end{equation}
and $b_\theta(x)$ is given by~\eqref{eq:aux:fun:b:L}.
Since we are assuming the hypotheses of Lemma~\ref{lem:iter},
we can use the intermediate estimates obtained in the proof
of that result (see Remark~\ref{rem:common}). In particular, the
estimates in~\eqref{eq:defor:al:L}, \eqref{eq:defor:h:L}, \eqref{eq:errL:L}
and the control on the distances
\[
\dist{\bar h_\te(\bar \TT_{\rho-\delta}),\partial \TT_{\hat\rho}}>0\,,
\qquad
\dist{\bar \alpha(\te),\partial A}>0\,,
\]
follow straightly.

Using property $\PV$ (of Lemma~\ref{lem:Lipschitz}) we readily obtain that
\begin{equation}\label{eq:Cerr:L}
\Lip{\Theta,\rho-\delta} (\partial_x e) \leq
\frac{\Lip{\Theta,\rho}(e)}{\delta}\,.
\end{equation}
Then, we control the function $\Delta_{\alpha}(\theta)$ in~\eqref{eq:Deltaa:L}
using property $\PIII$ and hypotheses
$\HI$, $\HII$
\begin{align*}
\Lip{\Theta}(\Delta_{\al}) \leq {} & 
\Lip{\Theta,\rho}(a) \Norm{1/\aver{b}}_{\Theta} 
+\norm{a}_{\Theta,\rho} \Norm{1/\aver{b}}_{\Theta}^2 \Lip{\Theta}(b)\\
\leq {} & 
\sighb \Lip{\Theta,\rho}(a) +
\norm{e}_{\Theta,\rho} \sighi (\sighb)^2 \Lip{\Theta,\rho}(b) \,.
\end{align*}
Now we control the Lipschitz bound
of the map $a_\te(x)$ in~\eqref{eq:Deltaa:L}, 
writing 
the rigid rotation as $R_\theta(x)=x+\theta$
and using $\PIII$, $\PIV$, and $\HIII$:
\begin{align}
\Lip{\Theta,\rho}(a) \leq {} & 
\Lip{\Theta,\rho}(e) \norm{1/\partial_x h}_{\Theta,\rho}+
\norm{e}_{\Theta,\rho} \Lip{\Theta,\rho}(1/\partial_x h\circ R) 
\nonumber \\
\leq {} & 
\Lip{\Theta,\rho}(e) \sighi +
\norm{e}_{\Theta,\rho} 
\left(
\Lip{\Theta,\rho}(1/\pd_x h) + \norm{\pd_x(1/\pd_x h)}_{\Theta,\rho}
\Lip{\Theta,\rho}(R)
\right)
\nonumber \\
\leq {} & \Lip{\Theta,\rho}(e) \sighi +
\norm{e}_{\Theta,\rho} (\sighi)^2 (\betaDh+\sighxx)\,, \label{eq:Liea:L}
\end{align}
where we used that $\Lip{\Theta,\rho}(R) = 1$,
\begin{equation}
\label{eq:Lpdh}
\Lip{\Theta,\rho}(1/\partial_x h) \leq 
\norm{1/\partial_x h}_{\Theta,\rho}^2 
\Lip{\Theta,\rho}(\partial_x h) < (\sighi)^2 \betaDh
\,, 
\end{equation}
and
\[
\norm{\pd_x(1/\pd_x h)}_{\Theta,\rho} \leq \norm{1/\pd_x h}_{\Theta,\rho}^2 \norm{\pd_{xx} h}_{\Theta,\rho} < (\sighi)^2 \sighxx\,.
\]
The Lipschitz control of the function $b_\theta(x)$ is analogous:
\begin{align}
\Lip{\Theta,\rho}(b) 
\leq {} &
\Lip{\Theta,\rho}(\pd_\al f_\al \circ h) \norm{1/\pd_x h}_{\Theta,\rho}
+ \norm{\pd_\al f_\al}_{\Theta,\hat \rho} \Lip{\Theta,\rho}(1/\pd_x h \circ R)
\nonumber     \\
\leq {} & (c_{\al\al} \betaa + c_{x\al} \betah) \sighi +
c_\al (\sighi)^2 (\betaDh+\sighxx)\,.\label{eq:Lieb:L}
\end{align}
Putting together the above computations we obtain~\eqref{eq:defor:al:L2}
with the constant $C_0^\sLip$ in~\eqref{eq:C0L}.

The control of $\eta_\theta$, given in~\eqref{eq:eta:L}, yields to
\[
\Lip{\Theta,\rho}(\eta) \leq
\Lip{\Theta,\rho}(a)+
\Lip{\Theta,\rho}(b) \norm{\Delta_\al}_{\Theta} 
+\norm{b}_{\Theta,\rho} \Lip{\Theta} (\Delta_\al)\,,
\]
where we used properties $\PI$ and $\PII$. 
Using the previous estimates in~\eqref{eq:control:a:b},
\eqref{eq:defor:al:L},
\eqref{eq:defor:al:L2},
\eqref{eq:Liea:L},
\eqref{eq:Lieb:L},
$\GI$,
$\HI$ and
$\HII$, and the expression of $\CLipI$ in~\eqref{eq:C1L}, we obtain
\[
\Lip{\Theta,\rho}(\eta) \leq
\CLipI  \norm{e}_{\Theta,\rho}
+ (1+c_\al \sighb \sighi) \sighi
\Lip{\Theta,\rho}(e)\,,
\]
where $\CLipI$ is given in~\eqref{eq:C1L}.
Then, we estimate the Lipschitz constant of the zero-average solution 
$\cR_\te \eta_\te(x)$ using $\PVI$ and $\GII$
\begin{equation}\label{eq:Reta:L}
\Lip{\Theta,\rho-\delta}(\cR \eta) \leq
\frac{\CLipII}{\gamma^2 \delta^{2\tau+1}} \norm{e}_{\Theta,\rho}
+ \frac{c_R (1+c_\al \sighb \sighi) \sighi
 }{\gamma \delta^\tau}\Lip{\Theta,\rho}(e)\,.
\end{equation}
Next, we control the average $\hat \varphi_0(\theta)=-\aver{\pd_x h_\te \cR\eta_\te}$
(we use $\HI$, $\HIII$ and the estimates~\eqref{eq:eta:c} and~\eqref{eq:Reta:L})
\begin{align*}
\Lip{\Theta}(\hat \varphi_0) \leq {} &
\Lip{\Theta,\rho}(\partial_x h) \norm{\cR \eta}_{\Theta,\rho-\delta}+
\norm{\partial_x h}_{\Theta,\rho} \Lip{\Theta,\rho-\delta}(\cR \eta) \\
\leq {} & \frac{\betaDh  c_R (1+c_\al \sighb \sighi) \sighi \gamma \delta^{\tau+1}
+ \sigh \CLipII}{\gamma^2 \delta^{2\tau+1}} \norm{e}_{\Theta,\rho} \\
& + \frac{\sigh c_R (1+c_\al \sighb \sighi) \sighi
 }{\gamma \delta^\tau}\Lip{\Theta,\rho}(e)\,,
\end{align*}
and we put together the previous two expressions as follows (using property $\PI$):
\begin{equation}\label{eq:varphi:L}
\Lip{\Theta,\rho-\delta}(\varphi) \leq 
\Lip{\Theta}(\hat \varphi_0) + \Lip{\Theta,\rho-\delta}(\cR \eta)
=: \frac{\CLipIII}{\gamma^2 \delta^{2\tau+1}} \norm{e}_{\Theta,\rho} +
\frac{C_1}{\gamma \delta^\tau} \Lip{\Theta,\rho}(e)\,,
\end{equation}
where $\CLipIII$ is given by~\eqref{eq:C2L}
and $C_1$ is given by~\eqref{eq:constC1}.

Then, the estimate in~\eqref{eq:defor:h:L2} is straightforward using that $\Delta_{h_{\theta}}(x)=\partial_x h_\theta(x) \varphi_\theta(x)$, 
the estimates \eqref{eq:varphi} and \eqref{eq:varphi:L}, together with
$\PII$, $\HI$ and $\HIII$.
In addition, we control the new maps
$\bar h_\theta(x)=h_\theta(x) + \Delta_{h_\theta}(x)$ as
\begin{align*}
\Lip{\Theta,\rho-\delta}(\bar h-\id) \leq {} & \Lip{\Theta,\rho}(h-\id)+\Lip{\Theta,\rho-\delta}(\Delta_h) \\
\leq {} &  \Lip{\Theta,\rho}(h-\id) +
\frac{\CLiph}{\gamma^2 \delta^{2\tau+1}} \norm{e}_{\Theta,\rho}
+ \frac{\sigh C_1}{\gamma \delta^\tau} \Lip{\Theta,\rho}(e) < \betah \,,
\end{align*}
where we used hypothesis~\eqref{eq:cond:lem:1:L}. The control of
$\partial_x \bar h_\theta(x)$
is analogous:
\begin{align*}
\Lip{\Theta,\rho-2\delta}(\partial_x \bar h) \leq {} &
\Lip{\Theta,\rho}(\partial_x h)+\Lip{\Theta,\rho-2\delta}(\partial_x \Delta_h) \\
\leq {} &  \Lip{\Theta,\rho}(\partial_x h) +
\frac{\CLiph}{\gamma^2 \delta^{2\tau+2}} \norm{e}_{\Theta,\rho}
+ \frac{\sigh C_1}{\gamma \delta^{\tau+1}} \Lip{\Theta,\rho}(e) < \betaDh \,,
\end{align*}
where we used hypothesis~\eqref{eq:cond:lem:2:L}. An analogous computation shows
that the smallness hypothesis in~\eqref{eq:cond:lem:3:L} guarantees that
$\Lip{\Theta}(\bar \al) < \betaa$.
Then, the smallness hypothesis in~\eqref{eq:cond:lem:4:L} guarantees that
$\norm{\pd_{xx}\bar h}_{\theta,\rho-2\delta} < \sighxx$. The computations are
similar to~\eqref{eq:mycond}:
\[
\norm{\pd_{xx}\bar \conj}_{\rho-2\delta} \leq \norm{\pd_{xx} \conj}_{\rho-2\delta} +
\norm{\pd_{xx}\Delta_\conj}_{\rho-2\delta} < \norm{\pd_{xx} \conj}_{\rho} + \frac{2\sigh
C_1}{\gamma\delta^{\tau+2}} \norm{e}_\rho < \sighxx\,.
\]

Finally, we control the new error of invariance
\[
\bar e_\theta(x) = f_{\bar \al(\theta)}(\bar h_\theta(x))-\bar h_\theta(x+\theta) =
\partial_x e_\theta(x) \varphi_\theta(x)+\Delta^2 f_\theta(x)
\]
using~\eqref{eq:varphi}, \eqref{eq:Cerr:L}, and~\eqref{eq:varphi:L}:
\begin{align}
\Lip{\Theta,\rho-\delta}(\bar e) \leq {} & \Lip{\Theta,\rho-\delta}(\partial_x e) \norm{\varphi}_{\Theta,\rho-\delta}
+\norm{\partial_x e}_{\Theta,\rho-\delta} \Lip{\Theta,\rho-\delta}(\varphi) +
\Lip{\Theta,\rho-\delta}(\Delta^2 f) \nonumber \\
\leq {} & 
\frac{\CLipIII}{\gamma^2 \delta^{2\tau+2}}\norm{e}_{\Theta,\rho}^2 + 
\frac{2C_1}{\gamma \delta^{\tau+1}}\norm{e}_{\Theta,\rho}\Lip{\Theta,\rho}(e)
+\Lip{\Theta,\rho-\delta}(\Delta^2 f)\,, \label{eq:errL:aux}
\end{align}
and it remains to control the map $\Delta^2 f_\theta(x)$, given by~\eqref{eq:D2F}.
This last term is controlled as follows
\begin{align*}
& \Lip{\Theta,\rho-\delta}(\Delta^2 f) \leq
\frac{1}{2} \Big(
\norm{\partial_{xx\al}f}_{\Theta,\rho} \betaa +
\norm{\partial_{xxx}f}_{\Theta,\rho} \betah
\Big) \norm{\Delta_h}_{\Theta,\rho-\delta}^2 \\
& \qquad +\norm{\partial_{xx} f}_{\Theta,\rho} \Lip{\Theta,\rho-\delta}(\Delta_h) \norm{\Delta_h}_{\Theta,\rho-\delta}\\
& \qquad + 
\Big(
\norm{\partial_{x\al\al}f}_{\Theta,\rho} \betaa
+
\norm{\partial_{xx\al}f}_{\Theta,\rho} 
\betah
\Big) \norm{\Delta_h}_{\Theta,\rho-\delta} \norm{\Delta_\al}_\Theta \\
& \qquad + \norm{\partial_{x\al} f}_{\Theta,\rho} 
\Big(
\Lip{\Theta,\rho-\delta}(\Delta_h) \norm{\Delta_\al}_{\Theta}
+
\norm{\Delta_h}_{\Theta,\rho-\delta} \Lip{\Theta}(\Delta_\al)
\Big) \\
& \qquad + 
\frac{1}{2} \Big(
\norm{\partial_{\al\al\al}f}_{\Theta,\rho} \betaa
+\norm{\partial_{x\al\al}f}_{\Theta,\rho} \betah
\Big) \norm{\Delta_\al}_{\Theta}^2 \\
& \qquad +\norm{\partial_{\al\al} f}_{\Theta,\rho} \Lip{\Theta}(\Delta_\al) \norm{\Delta_\al}_{\Theta,\rho}\,,
\end{align*}
where we used that
\begin{align*}
\sup_{t \in [0,1]} \Lip{\Theta,\rho-\delta}(h+t\Delta_h) \leq {} &
\Lip{\Theta,\rho}(h)+\Lip{\Theta,\rho-\delta}(\Delta_h) < \betah\,,\\
\sup_{t \in [0,1]} \Lip{\Theta}(\al+t\Delta_\al) \leq {} 
& \Lip{\Theta}(\al)+\Lip{\Theta}(\Delta_\al) < \betaa\,.
\end{align*}
Then, we use hypotheses $\GI$, $\HI$, $\HII$, $\HIII$, and the previous estimates for the
objects $\Delta_{h_\theta}(x)$ and $\Delta_{\al}(\theta)$.
Finally, we introduce this estimate into equation~\eqref{eq:errL:aux} and after
some tedious computations we obtain
that~\eqref{eq:errL:L2} holds using the constants $C_3$ and $\CLipIV$
given respectively by \eqref{eq:constC3} and \eqref{eq:C3L}.
\end{proof}

\begin{proof}[Proof of Theorem~\ref{theo:KAM:L}]
We reproduce the analysis of the convergence performed in
Theorem~\ref{theo:KAM}.
To this end, we consider the same sequences $\rho_s$, $\delta_s$ and
denote the corresponding objects at the $s$-th step by
$h_s$, $\al_s$ and $e_s$.

\bigskip
\emph{Existence of conjugations:} The conditions for the convergence of the sequence $\norm{e_s}_{\Theta,\rho_s}$
and the control of the norms of the objects are the same that we already obtained in
Theorem~\ref{theo:KAM}, so we must include them into~\eqref{eq:KAM:C1:L}. In the
following, we focus our attention in the control of the Lipschitz constants.
It is convenient to introduce the following weighted error:
\[
\cE_s:= \max\left\{\norm{e_s}_{\Theta,\rho_s}\,,\,\gamma \delta_{s}^{\tau+1}\Lip{\Theta,\rho_s}(e_s) \right\}\,,
\]
so that
\begin{equation}\label{eq:prop:w}
\norm{e_s}_{\Theta,\rho_s} \leq \cE_s\,,
\qquad
\Lip{\Theta,\rho_s}(e_s) \leq \frac{\cE_s}{\gamma \delta_{s}^{\tau+1}}\,.
\end{equation}
Using \eqref{eq:errL:L}, \eqref{eq:errL:L2} and the properties in \eqref{eq:prop:w},
we have
\begin{align}
\cE_s < {} & \max \Bigg\{
\frac{C_3\norm{e_{s-1}}_{\Theta,\rho_{s-1}}^2 }{\gamma^2 \delta_{s-1}^{2\tau}}
\,,\,
\frac{\gamma \delta_s^{\tau+1}\CLipIV \norm{e_{s-1}}_{\Theta,\rho_{s-1}}^2}{\gamma^3 \delta_{s-1}^{3\tau+1}} \nonumber \\
& \qquad\qquad\qquad\qquad
+
\frac{2C_3\gamma \delta_{s}^{\tau+1}}{\gamma^2 \delta_{s-1}^{2\tau}} \Lip{\Theta,\rho_{s-1}}(e_{s-1})
\norm{e_{s-1}}_{\Theta,\rho_{s-1}} 
\Bigg\} \nonumber
\\
< {} &  \frac{\max\{C_3\,,\,(\CLipIV+2C_3)(a_1)^{-\tau-1} \}}{\gamma^2 \delta_{s-1}^{2\tau}} \cE_{s-1}^2 =:
\frac{\CLipEE}{\gamma^2 \delta_{s-1}^{2\tau}} \cE_{s-1}^2\,. \label{eq:CLieEE}
\end{align}
Then, we reproduce the computations in~\eqref{eq:conv:err}-\eqref{eq:cond1}
asking for 
the condition
\begin{equation}\label{eq:cond1:L}
\mu := \frac{\CLipEE a_1^{2\tau}}{\gamma^2 \delta_0^{2\tau}}
\cE_0
<1\,
\end{equation}
where the constant $\CLipEE$ is evaluated at the worst case
$\delta_0$. We obtain that
\[
\cE_s
< \mu^{2^s-1} a_1^{-2\tau s}
\cE_0 \,.
\]
Then, we must check that the inequalities in~\eqref{eq:cond:lem:1:L},
\eqref{eq:cond:lem:2:L} and~\eqref{eq:cond:lem:3:L} are preserved along the iterative
procedure. Again, we assume that the constants $\CLipI$ and $\CLipIII$, given
by~\eqref{eq:C1L} and~\eqref{eq:C2L} respectively, are evaluated at
$\delta_0$. For example, we compute the following
\begin{align}
& \Lip{\Theta,\rho_s}(h_s-\id) + \frac{\CLiph}{\gamma^2 \delta_s^{2\tau+1}} \norm{e_s}_{\Theta,\rho_s} +
\frac{\sigh C_1}{\gamma \delta_s^\tau} \Lip{\Theta,\rho_s}(e_s)
\nonumber \\
& \qquad < \Lip{\Theta,\rho_s}(h_s-\id) + \frac{\CLiph + \sigh C_1}{\gamma^2 \delta_s^{2\tau+1}} \cE_s
\nonumber \\
& \qquad < \Lip{\Theta,\rho_0}(h_0-\id) + \sum_{j=0}^\infty \frac{\CLiph+\sigh C_1}{\gamma^2 \delta_j^{2\tau+1}} \cE_j \nonumber \\
& \qquad < \Lip{\Theta,\rho_0}(h_0-\id) + \frac{\CLipV\Sigma_{\mu,-1}}{\gamma^2 \delta_0^{2\tau+1}}
\cE_0
< \betah \label{eq:cond2:L}
\end{align}
where we used \eqref{eq:mysum}, \eqref{eq:prop:w} and introduced the constant
\begin{equation}\label{eq:C4L}
\CLipV := \CLiph
+ \sigh C_1 \,.
\end{equation}
As usual, the last inequality in~\eqref{eq:cond2:L} must be included in the
hypothesis~\eqref{eq:KAM:C1:L}. An analogous computation leads us to
\begin{align}
& \Lip{\Theta,\rho_s}(\partial_x h_s) + \frac{\CLiph}{\gamma^2 \delta_s^{2\tau+2}} \norm{e_s}_{\Theta,\rho_s} +
\frac{\sigh C_1}{\gamma \delta_s^{\tau+1}} \Lip{\Theta,\rho_s}(e_s)
\nonumber \\
& \qquad < \Lip{\Theta,\rho_0}(\pd_x h_0) + \frac{\CLipV \Sigma_{\mu,-2}}{\gamma^2 \delta_0^{2\tau+2}}
\cE_0< \betaDh \label{eq:cond3:L}
\end{align}
which is also included in~\eqref{eq:KAM:C1:L}. Finally, we have
\begin{equation}\label{eq:cond4:L}
\Lip{\Theta}(\al_s) < \Lip{\Theta}(\al_0)+ \frac{\CLipVI \Sigma_{\mu,\tau-1}}{\gamma \delta_0^{\tau+1}} 
\cE_0
< \betaa\,,
\end{equation}
which is also included in~\eqref{eq:KAM:C1:L}, where we introduced the constant
\begin{equation}\label{eq:C6L}
\CLipVI := C_0^\sLip \gamma \delta_0^{\tau+1} + \sighb \sighi\,.
\end{equation}
Putting together \eqref{eq:cond1:L},
\eqref{eq:cond2:L}, \eqref{eq:cond3:L}, \eqref{eq:cond4:L},
and \eqref{eq:KAM:C1:cond1}, \eqref{eq:KAM:C1:cond2},
\eqref{eq:KAM:C1:cond3}, \eqref{eq:KAM:C1:cond4}, \eqref{eq:KAM:C1:cond5}, and
using that $\rho/\delta=\rho_0/\delta_0=a_3$, we end up with
\begin{equation}\label{eq:frakC1:L}
\mathfrak{C}_1^\sLip := 
\left\{
\begin{array}{ll}
\max \left\{
(a_1 a_3)^{2\tau} \rho^2 \CLipEE,(a_3)^{\tau+1} C_4 \gamma
\rho^{\tau+1},\CLipVII
\right\}
& \mbox{if $\mu<1$} \,, \\
\infty & \mbox{otherwise} \,,
\end{array}
\right.
\end{equation}
where $\mu$ is given by~\eqref{eq:cond1:L}, 
$C_4$ is given by~\eqref{eq:const4} and we have introduced the new
constant
\begin{align}
\CLipVII := \max
\Bigg\{ &
\frac{\CLipV \Sigma_{\mu,-1} (a_3)^{2\tau+1} \rho}{\betah-\Lip{\Theta,\rho}(h-\id)}
,
\frac{\CLipV \Sigma_{\mu,-2} (a_3)^{2\tau}}{\betaDh-\Lip{\Theta,\rho}(\partial_x h)}
, 
\label{eq:C9L} \\
& \frac{\CLipVI \Sigma_{\mu,\tau-1} \gamma (a_3 \rho)^{\tau+1}}{\betaa-\Lip{\Theta}(\al)}
\Bigg\}\,. \nonumber
\end{align}
Notice that,
since $\CLipEE \geq C_3$, it turns out that
the condition \eqref{eq:KAM:C1:L} in Theorem~\ref{theo:KAM:L}
includes the condition \eqref{eq:KAM:C1} in Theorem~\ref{theo:KAM}. As a consequence,
we can apply Lemmata~\ref{lem:iter} and~\ref{lem:iter:L} again.
Therefore, the sequence
$
\cE_s
$
tends to zero when $s \rightarrow \infty$. The iterative scheme converges to a
family
$\theta \in \Theta \mapsto h_{\theta,\infty} \in \Difeo{\rho_\infty}$
and a function $\theta \in \Theta \mapsto \al_\infty(\theta) \in A$,
such that
\[
\map_{\al_\infty(\theta)}(\conj_{\theta,\infty}(x))=\conj_{\theta,\infty}(x+\rot)\,,\qquad
\forall \theta \in \Theta
\]
and $\conj_{\theta,\infty}$ also satisfies~\eqref{eq:normah}.

\bigskip
\emph{Measure of rotations:} As it was mentioned just after the statement of
Theorem~\ref{theo:KAM:L}, we can write
\[
\al_s(\theta)=\al_0(\theta)+\al_s(\theta)-\al_0(\theta)=\al_0(\theta)
+
\sum_{j=0}^{s-1}
\Delta_{\al_j}(\theta)\,.
\]
Then, reproducing the computation in~\eqref{eq:cond4:L},
we obtain
\[
\lip{\Theta}(\al_\infty) > \lip{\Theta}(\al_0) -
\sum_{j=0}^\infty \Lip{\Theta}(\Delta_{\al_{j}})
> 
\lip{\Theta}(\al_0)
- 
\frac{\CLipVI \Sigma_{\mu,\tau-1}}{\gamma \delta_0^{\tau+1}}
\cE_0
\,,
\]
so that the thesis $\TIV$ holds true by taking
\begin{equation}\label{eq:frakC2:L}
\mathfrak{C}_2^\sLip:=(a_3)^{\tau+1} \CLipVI \Sigma_{\mu,\tau-1} \,.
\end{equation}
This estimation of the Lipschitz constant from below is used to transport
the measure of the set $\Theta$ through the function $\theta \in \Theta \mapsto
\al_\infty(\theta) \in A$.
\end{proof}

\begin{remark}
Here we summarize how to compute constants
$\mathfrak{C}_1^\sLip$,
$\mathfrak{C}_2^\sLip$,
$\mathfrak{C}_2$,
$\mathfrak{C}_3$
of Theorem~\ref{theo:KAM:L}. Given fixed values
of the parameters
$\rho,\delta,\rho_\infty,\hat\rho$
and the distances
$\dist{\conj_\te(\bar \TT_\rho),\partial \TT_{\hat \rho}}$
and
$\dist{\al(\Theta),\partial A}$; the constants
$c_x$, $c_\al$, $c_{xx}$, $c_{x\al}$, $c_{\al\al}$,
$c_{xxx}$, $c_{xx\al}$, $c_{x\al\al}$, $c_{\al\al\al}$
in hypothesis $\GI$;
the constants $\gamma$ and $\tau$ in hypothesis $\GII$;
the constants $\sigma_1$, $\sigma_2$, $\sigma_3$ in hypothesis $\HI$;
the constant $\sigma_b$ in hypothesis $\HII$;
and the constants $\beta_0$, $\beta_1$, $\beta_2$ in hypothesis $\HIII$, are computed 
in the following order:
\begin{itemize}[leftmargin=5mm]
\item $a_1$, $a_2$, $a_3$ using~\eqref{eq:a1} and~\eqref{eq:a2a3}.
\item $C_1$, $C_2$, $C_3$ using~\eqref{eq:constC1}, \eqref{eq:constC2}
and~\eqref{eq:constC3}.
\item $C_0^\sLip$, $C_1^\sLip$, $C_2^\sLip$, $C_3^\sLip$, $C_4^\sLip$, $C_5^\sLip$
using~\eqref{eq:C0L}, 
\eqref{eq:C1L},
\eqref{eq:CLipII},
\eqref{eq:C2L},
\eqref{eq:ChL} and
\eqref{eq:C3L}.
\item $\kappa$, $\mu$ 
using~\eqref{eq:cond1} and~\eqref{eq:cond1:L}, and check that $\mu<1$ (abort the process otherwise).
\item
$\Sigma_{\kappa,\tau}$,
$\Sigma_{\kappa,2\tau}$,
$\Sigma_{\kappa,\tau-1}$,
$\Sigma_{\mu,-1}$,
$\Sigma_{\mu,-2}$,
$\Sigma_{\mu,\tau-1}$ using~\eqref{eq:mysum},
replacing $\norm{\cdot}_\rho$ by $\norm{\cdot}_{\Theta,\rho}$.
\item $C_4$ using~\eqref{eq:const4}.
\item $C_6^\sLip$, $C_7^\sLip$, $C_8^\sLip$, $C_9^\sLip$
using~\eqref{eq:CLieEE},  
\eqref{eq:C4L},  
\eqref{eq:C6L} and  
\eqref{eq:C9L}.
\item $\mathfrak{C}_2$,
$\mathfrak{C}_3$ using~\eqref{eq:frakC2C3}.
\item $\mathfrak{C}_1^\sLip$,
$\mathfrak{C}_2^\sLip$
using~\eqref{eq:frakC1:L} and~\eqref{eq:frakC2:L}.
\end{itemize}
\end{remark}

\section{On the verification of the hypotheses}\label{sec:hypo}

In this section we show a systematic approach,
tailored to be implemented in a computer-assisted proof,
to verify the assumptions of our a-posteriori theorems.
To do so, we perform an analytic study of the
hypotheses
with the goal of providing formulae that satisfy the following requirements: 
they are computable with a finite number operations, they give sharp bounds of the involved
estimates, and the computational time is fast:
all computations are performed with a complexity of order $N\log N$,
$N$ being the number of Fourier modes used to represent the conjugacies.
We focus in the hypotheses of Theorem~\ref{theo:KAM:L}, since the
discourse can be in fact simplified to deal with Theorem~\ref{theo:KAM}.

In Section~\ref{ssec:GIGII} we discuss the global hypotheses $\GI$ and $\GII$,
mainly how to obtain a suitable subset $\Theta$ of Diophantine numbers contained
in a given interval $B$ of rotation numbers, giving a sharp estimate
on the measure of $\Theta$.

In Section~\ref{ssec:HIHIIHIII} we discuss the hypotheses $\HI$, $\HII$ and
$\HIII$ which depend on the initial objects, which are taken as Fourier-Taylor polynomials.
Denoting
$\theta_0$ the center of the interval $B$,
we assume that $h_\theta(x)=h(x,\theta)$ is of the form
\begin{equation}\label{eq:hpol}
h(x,\theta) = x + \sum_{s = 0}^m h^{[s]}(x) (\theta-\theta_0)^s\,, \qquad
h^{[s]}(x) = \sum_{k=-N/2}^{N/2-1} \hat h^{[s]}_k \ee^{2\pi \ii k x}\,,
\end{equation}
where 
$\hat h_0^{[s]}=0$,
$\hat h^{[s]}_{-N/2}=0$ and $\hat h^{[s]}_k = (\hat h^{[s]}_{-k})^*$ otherwise,
and that $\al(\theta)$ is of the form
\begin{equation}\label{eq:alpol}
\al(\theta) = \sum_{s=0}^m \al^{[s]} (\theta-\theta_0)^s\,, \qquad \al^{[s]} \in \RR\,;
\end{equation}
for certain degree $m$ and $N=2^q$.
Notice that the symmetries in the Fourier coefficients
of $h$ have been selected so that the corresponding function is real-analytic
and satisfies the normalization condition
\[
\aver{h_\theta-\id}=0\,.
\]
We will see that the fact that these objects are chosen
to be polynomials plays an important role to obtain sharp values of the
constants $\sigma_1$, $\sigma_2$, $\sigma_3$, $\sigma_b$, $\beta_0$, $\beta_1$ and $\beta_2$.

In Section~\ref{ssec:error} we present the main result of this section, which allows us to obtain a
fine control of the norm of the error of conjugacy of the initial
family. This requires new tools since the corresponding error function,
denoted $e(x,\theta)$, is no longer a Fourier-Taylor polynomial, so we
combine the jet-propagation in the variable $\theta$ with the
control of the discrete Fourier approximation in the variable $x$.

For convenience, we briefly recall here some
standard notation used along this section, and an approximation theorem to control the error of the
discrete Fourier transform when approximating periodic functions.
Given a function $g:\TT \rightarrow \CC$, we consider its Fourier series
\[
g(x)= \sum_{k \in \ZZ} \hat g_k \ee^{2 \pi \ii k x} \,.
\]
Let us consider $N=2^q$ with $q\in \NN$, and the discretization $\{x_j\}$,
$x_j = j/N$, $0\leq j < N$, that defines a sampling $\{g_j\}$, with
$g_j=g(x_j)$.
Then, 
the discrete Fourier transform (DFT)
is 
\[
\tilde g_k = \frac{1}{N} \sum_{j=0}^{N-1} g_j \ee^{-2\pi \ii k x_j} =: \mathrm{DFT}_k(\{g_j\})
\]
and
the function $g(x)$ is approximated by the trigonometric polynomial
\[
\tilde g(x) := \sum_{k=-N/2}^{N/2-1} \tilde g_k \ee^{2\pi \ii k x}\,.
\]
We use the notation
\begin{align}
& \{\tilde g_k\} = \{\mathrm{DFT}_k(\{g_j\})\}\,, \label{eq:DFTF}\\
& \{g_j\} = \{\mathrm{DFT}_j^{-1}(\{\tilde g_k\})\}\,. \label{eq:DFTB}
\end{align}
Notice that formulae \eqref{eq:DFTF} and \eqref{eq:DFTB}
are exact 
if $g(x)$ is a trigonometric polynomial of degree 
at most $N$.
In this case, we write $\tilde g_k=\hat g_k$.

\begin{theorem}[See \cite{FiguerasHL17}]\label{theo:DFT}
Let $g:\TT_{\tilde\rho} \rightarrow \CC$ be an analytic and bounded function
in the complex strip $\TT_{\tilde\rho}$, with $\tilde\rho>0$. Let $\tilde g$ be
the DFT approximation of $g$ using $N$ nodes. Then
\begin{equation}\label{eq:prop:DFT1}
\abs{\tilde g_k - \hat g_k} \leq s_N(k,\tilde \rho) \norm{g}_{\tilde \rho}\,,
\end{equation}
\begin{equation}\label{eq:prop:DFT2}
\norm{\tilde g-g}_\rho \leq C_N(\rho,\tilde\rho) \norm{g}_{\tilde\rho}\,,
\end{equation}
for any $0\leq \rho \leq \tilde\rho$,
where
\begin{equation}\label{eq:sNrho}
s_{N}(k, \tilde\rho) :=  \frac{\ee^{-2\pi \tilde\rho N}}{1-\ee^{-2\pi \tilde\rho N}} \left( \ee^{2\pi \tilde\rho k}+ \ee^{-2\pi \tilde\rho k} \right)
\end{equation}
and
\begin{equation}\label{eq:CNrho}
C_{N}(\rho,  \tilde\rho)= S_ N^{1}(\rho, \tilde\rho) + S_ N^{2}(\rho, \tilde\rho)  + S_N^{3}(\rho, \tilde\rho)
\end{equation}
with
\begin{align*}
S_N^{1}(\rho, \tilde\rho) = {} &
 \frac{\ee^{-2\pi \tilde\rho  N} }{1-\ee^{-2\pi   \tilde\rho N }} \ 
 \frac{\ee^{-2\pi ( \tilde\rho+\rho)} + 1 }{\ee^{-2\pi ( \tilde\rho+\rho)} -1} 
 \left(1- \ \ee^{\pi( \tilde\rho+\rho)  N}\right) \,, \\
S_N^{2}(\rho, \tilde\rho) = {} & 
 \frac{\ee^{-2\pi \tilde\rho  N} }{1-\ee^{-2\pi   \tilde\rho  N }} \ 
  \frac{\ee^{2\pi ( \tilde\rho-\rho)} + 1 }{\ee^{2\pi ( \tilde\rho-\rho)} -1} 
  \left(1-\ \ee^{-\pi( \tilde\rho-\rho) N}\right) \,, \\
S_N^{3}(\rho, \tilde\rho) = {} & 
\frac{\ee^{2\pi ( \tilde\rho-\rho)} + 1}{\ee^{2\pi ( \tilde\rho-\rho)} -1} \ \ e^{-\pi( \tilde\rho-\rho)  N}\,.
\end{align*}
\end{theorem}

The following Fourier norm will be useful 
\[
\norm{g}_{\rho}^\cF
:= \sum_{k \in \ZZ} \abs{\hat g_k} \ee^{2\pi |k| \rho}\,,
\]
since it can be evaluated with a finite amount of computations if $g(x)$
is a trigonometric polynomial. Notice that 
$\|g\|_\rho\leq \|g\|^{\cF}_\rho$.

For convenience,
we use the suitable language of interval analysis, 
with the only aim of controlling truncation and discretization errors.
In particular, for any closed interval $Z \subset \RR$ we use the standard notation
\begin{equation}\label{eq:Intbound}
Z=[\underline{Z},\overline{Z}]\,, \qquad
\mathrm{rad}(Z)=\frac{\overline{Z}-\underline{Z}}{2}
\end{equation}
for the boundaries and the radius of an interval.
The error produced when evaluating the proposed expressions using a computer (with finite precision
arithmetics) is easily controlled performing computations with
interval arithmetics.

\DeclarePairedDelimiter\ceil{\lceil}{\rceil}
\DeclarePairedDelimiter\floor{\lfloor}{\rfloor}

\subsection{Controlling the global hypotheses $\GI$ and $\GII$} \label{ssec:GIGII}

Obtaining
the bounds on the derivatives of the map, $\GI$, is problem 
dependent and does not suppose a big challenge. If the map $f_\alpha$ 
is given in an explicit form then the bounds can be obtained
by hand, as we illustrate in Section~\ref{sec:example:Arnold} with an example.

The global hypothesis $\GII$ has to do with finding (positive measure) sets of Diophantine numbers in a closed
interval.

\begin{lemma}\label{sec:lemma:Dio}
Given an interval $B= [\underline{B},\overline{B}] \subset \RR$,
constants $\gamma<\frac12$ and $\tau>1$, and
$Q\in \NN$ such that $\frac{2}{Q}\leq \overline{B}-\underline{B}$, then 
the relative measure of $\Theta=B\cap \cD(\gamma,\tau)$ satisfies
\begin{equation}\label{eq:measT}
\frac{\meas(\Theta)}{\meas(B)}
\geq
1-\frac{4\gamma }{(\tau-1)Q^{\tau-1}}
-
\sum_{q=1}^Q
\sum_{\substack{p= \floor{\underline{B} q} \\ \mathrm{gcd}(p,q)=
1}}^{\ceil{\overline{B} q}}
\Delta(p,q)\,,
\end{equation}
where
\begin{equation}\label{eq:DeltaT}
\Delta(p,q) = 
\left\{
\begin{array}{ll}
\max\left( \frac{p}{q}+\frac{\gamma}{q^{\tau+1}} -\underline{B},0\right) & \mbox{if $p= \floor{\underline{B} q}$\,,} \\
\max\left( \overline{B}-\frac{p}{q}+\frac{\gamma}{q^{\tau+1}},0\right) & \mbox{if $p= \ceil{\overline{B} q}$\,,} \\
\min\left(\overline{B},\frac{p}{q}+\frac{\gamma}{q^{\tau+1}}\right) -
\max\left(\underline{B},\frac{p}{q}-\frac{\gamma}{q^{\tau+1}}\right)& \mbox{otherwise}\,.
\end{array}
\right.
\end{equation}
Here we use the notation $\ceil{\cdot}$ and $\floor{\cdot}$ for the ceil and floor functions, respectively.
\end{lemma}

\begin{proof}
Since
$\gamma<\frac12$, 
the
$(p,q)$-resonant sets
\[
\textrm{Res}_{p,q}(B,\gamma,\tau) =
\left\{
\theta \in B\,:\, |q \theta-p|< \frac{\gamma}{q^\tau}
\right\}\,,
\]
for 
each fixed $q>0$, are 
pairwise disjoint. Notice also that, for any $k\in \NN$, 
$\textrm{Res}_{kp,kq}(B,\gamma,\tau) \subset \textrm{Res}_{p,q}(B,\gamma,\tau)$, so that 
 that the full resonant set of type $(\gamma,\tau)$ is
\[
\textrm{Res}(B,\gamma,\tau) =
\bigcup_{\mathrm{gcd}(p,q)= 1} \textrm{Res}_{p,q}(B,\gamma,\tau) \,,
\]
and so, finding a lower bound of $\meas(\Theta)$
is equivalent to finding an upper bound of $\meas(\textrm{Res}(B,\gamma,\tau))$.

Given a fixed number $Q$,
we consider the disjoint union
\[
\textrm{Res}(B,\gamma,\tau)
= \textrm{Res}_{\leq Q}(B,\gamma,\tau)
    \cup
    \textrm{Res}_{>Q}(B,\gamma,\tau)
\]
where $\textrm{Res}_{\leq Q}(B,\gamma,\tau)$ and
$\textrm{Res}_{>Q}(B,\gamma,\tau)$ are, respectively, the sets of resonances
with denominator $q$ satisfying $q\leq Q$ and $q>Q$. The measure of the first
set is controlled as 
\begin{align*}
\meas (\textrm{Res}_{\leq Q}(B,\gamma,\tau)) \leq {} &  \sum_{q=1}^Q
\sum_{\substack{p= \floor{\underline{B} q} \\ \mathrm{gcd}(p,q)=
1}}^{\ceil{\overline{B} q}} \meas (\textrm{Res}_{p,q}(B,\gamma,\tau)) \\
\leq {} &
(\overline{B}-\underline{B})
\sum_{q=1}^Q
\sum_{\substack{p= \floor{\underline{B} q} \\ \mathrm{gcd}(p,q)=
1}}^{\ceil{\overline{B} q}} \Delta(p,q) \,,
\end{align*}
where $\Delta(p,q)$ is given in~\eqref{eq:DeltaT}.
The measure of the second set
is controlled as
\[
\begin{split}
\meas (\textrm{Res}_{>Q}(B,\gamma,\tau)) 
\leq & \sum_{q= Q+1}^\infty \frac{2\gamma}{q^{\tau+1}} (\ceil{\overline{B} q}-\floor{\underline{B} q})\\
\leq & \sum_{q= Q+1}^\infty \frac{2\gamma}{q^{\tau+1}} (\overline{B} q-\underline{B} q+2) \\
\leq & 2\gamma \left(\frac{\overline{B}-\underline{B}}{(\tau-1)Q^{\tau-1}} 
+ \frac{2}{\tau Q^\tau}  \right).
\end{split}
\]
Finally, since $\frac{2}{Q}\leq \overline{B}-\underline{B}$, we get the upper bound 
\[
\meas (\textrm{Res}_{>Q}(B,\gamma,\tau)) \leq  \frac{4\gamma (\overline{B}-\underline{B})}{(\tau-1)Q^{\tau-1}}.
\]
Then, the bound~\eqref{eq:measT} holds by combining both estimates.
\end{proof}

\begin{remark}\label{rem:meas:gbl}
When $B=[0,1]$ we have
$\Delta(p,q)=\frac{2 \gamma}{q^{\tau+1}}$ for every $p$.
Then, 
taking $Q\rightarrow \infty$
we obtain a Dirichlet series as a lower bound
\[
\meas(\Theta) = \frac{\meas(\Theta)}{\meas(B)}
\geq
1- 2\gamma \sum_{q=1}^\infty \frac{\phi(q)}{q^{\tau+1}}
= 1-2\gamma \frac{\zeta(\tau)}{\zeta(\tau+1)},  
\]
where $\phi$ is the Euler function and $\zeta$ is the Riemann zeta function.
\end{remark}

\subsection{Controlling the hypotheses $\HI$, $\HII$ and $\HIII$}\label{ssec:HIHIIHIII}

Given an interval $B$, centered at $\theta_0$,
we consider $h_\theta(x)=h(x,\theta)$ 
and $\al(\theta)$ given by~\eqref{eq:hpol}
and~\eqref{eq:alpol}, respectively.
In this section we present a procedure
to control the hypothesis $\HI$, $\HII$ and $\HIII$ corresponding to these objects.
To this end, it is convenient to introduce some notation to enclose
the dependence of the variable $\theta$.

\begin{definition}\label{def:enclo}
Given a function of the form
\[
F(x,\theta) = \sum_{s \geq 0} F^{[s]}(x) (\theta-\theta_0)^s\,, \qquad
F^{[s]}(x) = \sum_{k \in \ZZ} \hat F^{[s]}_k \ee^{2\pi \ii k x}\,,
\]
we introduce the enclosing interval function and its formal derivative as follows:
\begin{equation}\label{eq:EncloF}
F_B(x) := \sum_{k \in \ZZ} \hat F_{B,k} \ee^{2\pi \ii k x}\,,
\qquad
F_B'(x) := \sum_{k \in \ZZ} (2\pi \ii k) \hat F_{B,k} \ee^{2\pi \ii k x}\,,
\end{equation}
where $\hat F_{B,k}$ are given by
\[
\hat F_{B,k} := \sum_{s \geq 0} \hat F^{[s]}_k (B-\theta_0)^s
= \left\{
\sum_{s\geq 0} \hat F^{[s]}_k (\theta-\theta_0)^s\,:\,\theta \in B
\right\}\,.
\]
Abusing notation we write
\[
\norm{F_B}_{\rho} := 
\max_{x \in \TT_\rho} \overline{\abs{\sum_{k\in \ZZ} \hat F_{B,k} \ee^{2\pi \ii k x}}}
\,,
\qquad
\norm{F_B}_{\rho}^\cF := 
\overline{\sum_{k\in \ZZ} \abs{\hat F_{B,k}} \ee^{2\pi 
|k| \rho}}\,.
\]
\end{definition}

Using the enclosing interval function associated 
to a Fourier-Taylor polynomial of the form
\[
F(x,\theta) = \sum_{s = 0}^m F^{[s]}(x) (\theta-\theta_0)^s\,, \qquad
F^{[s]}(x) = \sum_{k =-N/2}^{N/2-1} \hat F^{[s]}_k \ee^{2\pi \ii k x}\,,
\]
we introduce the following notation:
\[
\mathfrak{M}_{B,\rho}(F) := \max_{j=0,\ldots,N-1} \left\{\Abs{\overline{\mathrm{DFT}^{-1}_j(\{ \hat F_{B,k} \ee^{2\pi \ii k\rho}\})}}\right\} 
+ \frac{1}{2N} {\norm{F_B'}_{\rho}^\cF}\,.
\]
and
\[
\mathfrak{m}_{B,\rho}(F) := \min_{j=0,\ldots,N-1} \left\{\Abs{\underline{\mathrm{DFT}^{-1}_j(\{ \hat F_{B,k} \ee^{2\pi \ii k\rho}\})}}\right\} 
- \frac{1}{2N} {\norm{F_B'}_{\rho}^\cF}\,.
\]

\begin{proposition}\label{prop:H1check}
Consider the setting of Theorem~\ref{theo:KAM:L}, taking
$h(x,\theta)$
and
$\al(\theta)$ 
of the form~\eqref{eq:hpol}
and~\eqref{eq:alpol}, respectively, and
assume that
\begin{equation}\label{eq:Lem:Hyp}
\hat \rho > \rho + \mathfrak{M}_{B,\rho}(h-\id)\,,
\quad
\mathfrak{M}_{B,\rho}(\pd_x h-1)<1\,,
\quad
\mathfrak{m}_{B,\rho}(\pd_x h)> 0\,.
\end{equation}
Then:
\begin{enumerate}
\item Hypothesis $\HI$ holds by taking 
\begin{align*}
\sigma_1 > {} & \mathfrak{M}_{B,\rho}(\pd_x     h) ,\\
\sigma_2 > {} & 1/\mathfrak{m}_{B,\rho}(\pd_x   h) ,\\
\sigma_3 > {} & \mathfrak{M}_{B,\rho}(\pd_{xx}  h)\,.
\end{align*}
\item 
Consider the function
\[
b(x,\theta)=\frac{\pd_\al f_{\al(\theta)} (h(x,\theta))}{\pd_x h (x+\theta,\theta)}\,,
\]
and let $\{b_B(x_j)\}$ be the corresponding enclosing function 
evaluated in the grid $x_j=j/N$. Assume that
\begin{equation}\label{eq:hyp_bt}
c_b := \frac{s_N(0,\rho)
c_\al\overline{\abs{1/\widetilde{b}_{B,0}}}}{\mathfrak{m}_{B,\rho}(\pd_x
h)} < 1\,,
\end{equation}
where $s_N(0,\rho)$ is given in~\eqref{eq:sNrho} and
\[
\widetilde{b}_{B,0} = \frac{1}{N} \sum_{j=0}^{N-1} b_B(x_j)\,.
\]
Then, hypothesis $\HII$ holds by taking 
\[
\sigma_b > \frac{\overline{\abs{1/\widetilde{b}_{B,0}}}}{1-c_b}\,.
\]
\item Assume that 
the interval 
$\al'(B)$
does not
contain $0$.
Then, hypothesis $\HIII$ holds by taking
\begin{align*}
\beta_0 > {} & \mathfrak{M}_{B,\rho}(\pd_\theta h)\,, \\
\beta_1 > {} & \mathfrak{M}_{B,\rho}(\pd_{x\theta} h)\,, \\
\beta_2 > {} & \overline{\al'(B)}\,.
\end{align*}
Furthermore, we observe that 
$\underline{\al'(B)} < \lip{\Theta}(\al)$.
\end{enumerate}
\end{proposition}

\begin{remark}
Since $h(x,\theta)$ is a Fourier-Taylor polynomial, we could directly use the Fourier norm
to produce
\[
\sigma_1 > 1+\sum_{s=0}^m \norm{\pd_x h^{[s]}}_\rho^\cF \ \mathrm{rad}(B)^s\,,
\qquad
\sigma_3 > \sum_{s=0}^m \norm{\pd_{xx} h^{[s]}}_\rho^\cF \ \mathrm{rad}(B)^s\,.
\]
This approach not only produces substantial overestimation (which is propagated in the
KAM constants $\mathfrak{C}_1$, $\mathfrak{C}_2$, etc), but does not give information
to control $\sigma_2$. Notice that the overestimation (produced by the Fourier norm)
of the derivative $F_B'$ in
$\mathfrak{M}_{B,\rho}(F)$ and $\mathfrak{m}_{B,\rho}(F)$
is mitigated both by the factor $1/2N$ and by enclosing the dependence of $\theta$
(since cancellations are taken into account).
\end{remark}

\begin{proof}
We first observe that, since $h(x,\theta)$ is a Fourier-Taylor polynomial,
it is indeed analytic in $B \supset \Theta$,
so we consider the bounds
\[
\norm{\pd_x h}_{\Theta,\rho} \leq \norm{\pd_x h}_{B,\rho}\,,\quad
\norm{1/\pd_x h}_{\Theta,\rho} \leq \norm{1/\pd_x h}_{B,\rho}\,,\quad
\norm{\pd_{xx} h}_{\Theta,\rho} \leq \norm{\pd_{xx} h}_{B,\rho}\,.
\]
To control $\norm{\pd_x h}_{B,\rho}$ and $\norm{\pd_{xx} h}_{B,\rho}$, we use
the maximum 
modulus principle for analytic functions. By hypothesis, the functions are real-analytic,
so it suffices to consider
one component of the boundary, say
$\{\im(x)=\rho\}$.

Using the enclosing operation in Definition~\ref{def:enclo} we 
reduce
the discussion to manipulate formal Fourier series
$\pd_x h_B(x)$ and $\pd_{xx} h_B(x)$,
with intervalar coefficients, that include the dependence of the variable $\theta$.
Hence, to estimate the maximum of a function $F(x,\theta)$ in $\{\im(x)=\rho\} \times B$ we
construct the bound $\mathfrak{M}_{B,\rho}(F)$
as follows:
\begin{itemize}[leftmargin=5mm]
\item The restriction of the enclosing function $F_B(x)$
to the boundary is obtained by multiplying each $k^\mathrm{th}$ Fourier coefficient by
$\ee^{2\pi \ii k\rho}$.
\item The evaluation of $F_B(x + \ii \rho)$ in the uniform grid of $N$ intervals
is performed using \textrm{DFT}.
\item The maximum of $F_B(x + \ii \rho)$ in each interval of length $1/N$, centered at the
grid points, is bounded above by the value of the function at the grid plus
a global bound of the derivative.
\item The bound of the mentioned derivative is obtained using the immediate inequality
\[
\norm{F_B'}_{\rho} 
\leq \norm{F_B'}_{\rho}^\cF
\,.
\]
\end{itemize}
The above discussion allows us to control $\sigma_1$ and $\sigma_3$.

To control $\sigma_2$ we perform an analogous argument for the minimum
modulus principle. Notice that the condition
$\mathfrak{M}_{B,\rho}(\pd_xh-\id)<1$ ensures that the function $\pd_x h(x,\theta)$ is non-zero at all
points in $\TT_\rho \times B$. The fact that $\mathfrak{m}_{B,\rho}(\pd_x h)> 0$
ensures that we can take $\sigma_2<\infty$.

As the last condition in hypothesis $\HI$, we must see that for every $\theta-\theta_0 \in B$ the map $h(x,\theta)$ satisfies
\begin{equation}\label{eq:disp:prop}
\dist{h(\bar \TT_\rho,\theta),\pd \TT_{\hat\rho}} > 0\,.
\end{equation}
To this end, we compute
\[
\max_{x\in \bar \TT_\rho} \max_{\theta \in B}
\abs{\im (h(x,\theta))} \leq \rho + \norm{h-\id}_{B,\rho} 
\leq \rho + \mathfrak{M}_{B,\rho}(h-\id)
\]
and the inequality~\eqref{eq:disp:prop} holds from the first assumption in~\eqref{eq:Lem:Hyp}.
This completes item (1).

Regarding item (2),
let us recall that we are interested in controlling
$
\norm{1/\aver{b}}_{\Theta}
$
where $\aver{b}(\theta)$ is the actual average with respect
to $x$. Notice that
\[
\aver{b}(\theta) = \sum_{s \geq 0}
\aver{b^{[s]}} (\theta-\theta_0)^s
\in
\int_0^1 b_B(x) \dif x =: \aver{b_B}
\]
and so
\[
\norm{1/\aver{b}}_{\Theta} \leq \overline{\abs{1/\aver{b_B}}}\,.
\]
Using the notation in the statement, and Theorem~\ref{theo:DFT}, we have
\[
\overline{\Abs{\widetilde{b}_{B,0}-\aver{b_B}}}
\leq s_N(0,\rho) \norm{b}_{B,\rho} \leq s_N(0,\rho)\norm{f_\al}_{A,\hat\rho}\norm{1/\pd_x h}_{B,\rho}
\leq \frac{s_N(0,\rho)\ c_\al}{\mathfrak{m}_{B,\rho}(\pd_x h)}\,.
\]
Then, we compute
\[
\overline{\abs{1/\aver{b_B}}} \leq
\frac{\overline{\Abs{1/\widetilde{b}_{B,0}}}}
{1-  \overline{\Abs{1/\widetilde{b}_{B,0}}}
     \overline{\Abs{\widetilde{b}_{B,0}-\aver{b_B}}}}
\leq
\frac{\overline{\Abs{1/\widetilde{b}_{B,0}}}}
{1- c_b}\,,
\]
where we used~\eqref{eq:hyp_bt}.

Finally, item (3) follows reproducing the argument for item (1), but controlling
the Lipschitz norms in terms of the norm of the corresponding derivative
with respect to $\theta$.
\end{proof}

\subsection{Controlling the error of conjugacy}\label{ssec:error}

Given an interval $B$, centered at $\theta_0$,
we consider again $h_\theta(x)=h(x,\theta)$ 
and $\al(\theta)$ given by~\eqref{eq:hpol}
and~\eqref{eq:alpol}, respectively.
In this section we propose suitable (sharp and computable) estimates to
control the norm $\norm{\cdot}_{\Theta,\rho}$ of the Fourier-Taylor series
\begin{equation}\label{eq:error:Lip}
e(x,\theta) 
= \sum_{s\geq 0} e^{[s]}(x) (\theta-\theta_0)^s
:= f_{\al(\theta)}(h(x,\theta)) - h(x+\theta,\theta)\,.
\end{equation}
Notice that we have to compose $f_\al(x)$ with the objects 
$h(x,\theta)$ and $\al(\theta)$ given by~\eqref{eq:hpol}
and~\eqref{eq:alpol}.
Assuming that the family $f_\al(x)=f(x,\al)$ is $C^\infty$ in $\al$, 
we can express the composition as follows
\begin{equation}\label{eq:fps}
F(x,\theta)=f(h(x, \theta),\al(\theta)) = 
x+\sum_{s\geq 0} 
F^{[s]}(x) (\theta-\theta_0)^s\,,
\end{equation}
where 
\[
F^{[0]}(x)=\mathfrak{F}_0(x+h^{[0]}(x),\al^{[0]})=f_{\al^{[0]}}(x+h^{[0]}(x))-x\,,
\]
and the remaining coefficients, for $s\geq 1$,
are given by recurrence formulae
\begin{equation}\label{eq:recurrences}
\begin{split}
F^{[s]}(x) = \mathfrak{F}_s\Big(&x+h^{[0]}(x),\al^{[0]};
                                h^{[1]}(x),\ldots,h^{[s]}(x),
                                \al^{[1]},\ldots,\al^{[s]},\\
                                &F^{[1]}(x),\ldots,F^{[s-1]}(x)\Big)\,.
\end{split}
\end{equation}
Notice that the recurrences $\mathfrak{F}_s$ are explicit in terms of
Fa\`a di Bruno formulae or, if the function $f_{\al}(x)$ is elementary, using 
Automatic Differentiation rules (see~\cite{Knuth97}). In particular,
formula~\eqref{eq:recurrences} depends polynomially with respect to
$h^{[1]}(x), \ldots, h^{[s]}(x)$,
$\al^{[1]}, \ldots, \al^{[s]}$,
$F^{[1]}(x), \ldots, F^{[s-1]}(x)$.

Furthermore, a natural way to enclose the power series~\eqref{eq:fps}
is the truncated Taylor model (recall that $m$ is the fixed
order in $\theta-\theta_0$ of the initial objects $h(x,\theta)$ and $\al(\theta)$)
\[
F(x,\theta) \in x+\sum_{s=0}^m F^{[s]}(x)
(\theta-\theta_0)^s + F^{[m+1]}_B(x) [-1,1]\ \mathrm{rad}(B)^{m+1}\,,
\]
where $F^{[m+1]}_B(x)$ is obtained evaluating the same recurrences
\begin{equation}\label{eq:recurrences:mp1}
\begin{split}
F_B^{[s]}(x) = \mathfrak{F}_s\Big(&
x+h_B^{[0]}(x),\al_B^{[0]};
h_B^{[1]}(x),\ldots,h_B^{[s]}(x),
\al_B^{[1]},\ldots,\al_B^{[s]},\\
&F_B^{[1]}(x),\ldots,F_B^{[s-1]}(x)\Big)\,,
\end{split}
\end{equation}
for $1\leq s\leq m+1$,
with the fattened 
objects 
\begin{equation}\label{eq:fattenedhal}
h_B^{[s]}(x) =\frac{1}{s!} \left\{\frac{\pd^s h(x, \theta)}{\pd
\theta^s}:\theta\in B\right\} \,,
\qquad
\al_B^{[s]} =\frac{1}{s!} \left\{\frac{\dif^s \al(\theta)}{\dif
\theta^s}:\theta\in B\right\} \,.
\end{equation}
See \cite{Tucker11} for further details.

In the following theorem we propose an explicit 
estimate for the norm of
the error \eqref{eq:error:Lip} using
the above 
idea. 
A major obstacle is that the space of
trigonometric polynomials of degree at most $N$
is not an algebra. This is overcome by combining
recurrences \eqref{eq:recurrences}
and~\eqref{eq:recurrences:mp1}
with control on the discretization
error in Fourier space. Thus, we obtain an additional source of error
that, remarkably, is estimated using recursive formulae that depend
only on the family $f_\al(x)$. 

\begin{theorem}\label{theo:Finite}
Consider the setting of Theorem~\ref{theo:KAM:L}, taking
$h(x,\theta)$ and $\al(\theta)$ of the form~\eqref{eq:hpol}
and~\eqref{eq:alpol}, respectively, and
assume that, given $r>\hat\rho$, we have
$f_{\al}(x) \in \Anal(\TT_{r})$ and the maps are $C^{m+1}$ with respect to $\al$.
Then, for any 
$\tilde \rho > \rho$ such that
\begin{equation}\label{eq:condError}
r> \tilde \rho + 
\mathfrak{M}_{\theta_0,\tilde\rho}(h^{[0]})
\,,
\end{equation}
the error \eqref{eq:error:Lip} satisfies
\begin{align*}
\norm{e}_{\Theta,\rho} \leq {} & 
\sum_{s=0}^m \norm{\widetilde{e}^{[s]}}_{\rho}^\cF \ \mathrm{rad}(B)^s
+ C_T \ \mathrm{rad}(B)^{m+1}
+ C_F \ C_N(\rho,\tilde\rho)\,,\\
\Lip{\Theta,\rho}(e) \leq {} & 
\sum_{s=1}^{m} s\norm{\widetilde{e}^{[s]}}_{\rho}^\cF \ \mathrm{rad}(B)^{s-1}
+ (m+1) C_T \ \mathrm{rad}(B)^{m}
+ C_F' \ C_N(\rho,\tilde\rho)\,,
\end {align*}
where $\widetilde{e}^{[s]}(x)$ is the discrete Fourier approximation
given by
\begin{equation}\label{eq:DFTe}
\widetilde{e}^{[s]}_k = \mathrm{DFT}_k(\{e^{[s]}(x_j)\})\,,
\end{equation}
with $x_j=j/N$, 
and the computable constants $C_T$, $C_F$, $C_F'$ that depend
on $m,B,\rho,\tilde\rho$ and the initial objects.
\end{theorem}

\begin{remark}
Note that the functions 
$\widetilde{e}^{[s]}(x)$, for $0\leq s \leq m$, are expected
to be small if the candidates $h(x,\theta)$ and
$\al(\theta)$ are good enough approximations of the Lindstedt series at $\theta_0$.
Moreover, the constant $C_N(\rho,\tilde\rho) =O(\ee^{-\pi N(\tilde\rho-\rho)})$
can be taken very small for a suitable choice of $\tilde\rho$. Hence,
the limiting factor of the estimate is the term $C_T \ \mathrm{rad}(B)^{m+1}$ (see~\eqref{eq:finalstep2}).
\end{remark}

\begin{proof}
The error \eqref{eq:error:Lip} can be enclosed as
\[
e(x,\theta) \in \sum_{s=0}^{m} e^{[s]}(x) (\theta-\theta_0)^s 
+ e_{B}^{[m+1]}(x) [-1,1]\ \mathrm{rad}(B)^{m+1}\,,
\]
which yields the control
\[
\norm{e}_{\Theta,\rho} \leq \sum_{s=0}^m \norm{e^{[s]}}_\rho \ \mathrm{rad}(B)^s
+ 
\Norm{e_{B}^{[m+1]}}_\rho
\ \mathrm{rad}(B)^{m+1}\,.
\]
Since the functions $e^{[s]}(x)$ and $e_{B}^{[m+1]}(x)$ have
infinitely many harmonics, we approximate them using suitable
trigonometric polynomials $\tilde e^{[s]}(x)$ (\emph{Step 1}) and
$\tilde e_{B}^{[m+1]}(x)$ (\emph{Step 2}), thus obtaining the bound
\begin{align*}
\norm{e}_{\Theta,\rho} \leq {} & 
\sum_{s=0}^m \norm{\tilde e^{[s]}}_\rho \ \mathrm{rad}(B)^s + 
\Norm{\tilde e_{B}^{[m+1]}}_\rho \ \mathrm{rad}(B)^{m+1} \\
&+\sum_{s=0}^m \norm{e^{[s]}-\tilde e^{[s]}}_\rho \ \mathrm{rad}(B)^s + 
\Norm{e_{B}^{[m+1]}-\tilde e_B^{[m+1]}}_\rho \ \mathrm{rad}(B)^{m+1} 
\,.
\end{align*}
Then, we deal with the error committed by 
approximating $e^{[s]}(x)$ and $e_{B}^{[m+1]}(x)$ with 
$\tilde e^{[s]}(x)$ and $\tilde e_{B}^{[m+1]}(x)$ (\emph{Step 3}). 
After controlling the norm $\norm{e}_{\Theta,\rho}$, we use
that the function is indeed smooth in $B \supset \Theta$ and
control $\Lip{\Theta,\rho}(e)$ (\emph{Step 4}).

\bigskip
\emph{Step 1:} 
The Taylor coefficients
of the first term in~\eqref{eq:error:Lip},
$F(x,\theta)=f_{\al(\theta)}(h(x,\theta))$, satisfy the recurrences
\eqref{eq:recurrences}. In particular, we evaluate them pointwise in
the grid $x_j = j/N$, $0\leq j < N$, thus obtaining
$\{F^{[s]}(x_j)\}$. Notice that the evaluations 
\[
\{h^{[s]}(x_j)\} = \{\mathrm{DFT}_j^{-1}(\{\tilde h_k^{[s]}\})\}\,,
\]
are exact,
since ${h}^{[s]}(x)$ are trigonometric polynomials. 

For the second term in~\eqref{eq:error:Lip}, 
$H(x,\te)=h(x+\theta,\theta)$, we have
\begin{align}
H(x,\theta)= {} & x+\sum_{s\geq 0} H^{[s]}(x) (\theta-\theta_0)^s\\
:= {} & x+\theta+\sum_{s=0}^m 
h^{[s]} (x+\theta_0+\theta-\theta_0) (\theta-\theta_0)^s \nonumber \\
= {} & 
x +\theta+
\sum_{s =0}^m
\left(
\sum_{j\geq 0} \frac{1}{j!} \frac{\dif^j h^{[s]}(x+\theta_0) }{\dif x^j} (\theta-\theta_0)^j
\right) (\theta-\theta_0)^s \nonumber \\
= {} & 
x +\theta_0 + (\theta-\theta_0) +
\sum_{s\geq 0} \left(
\sum_{\substack{j=0\\j\geq s-m}}^s \frac{1}{j!} \frac{\dif^j h^{[s-j]}(x+\theta_0) }{\dif x^j}
\right) (\theta-\theta_0)^s \,. \nonumber
\label{eq:partII} 
\end{align}
Notice that the Fourier coefficients of 
$H^{[s]}(x)$, $\widehat{H}_k^{[s]}$, are just finite linear combinations 
(obtained from derivatives and shifts of angle $\theta_0$)
of the Fourier coefficients of $h^{[l]}(x)$, $0\leq
l \leq s$.

Putting together the two terms, 
we obtain (notice that the affine part in the $[0]^\mathrm{th}$ coefficient
cancels out)
\[
e^{[s]}(x)= F^{[s]}(x)-H^{[s]}(x)\,,\qquad 0 \leq s \leq m.
\]
Therefore, we obtain the Fourier coefficients of
the approximations $\tilde e^{[s]}$ as
\[
\widetilde{e}^{[s]}_k = \mathrm{DFT}_k(\{F^{[s]}(x_j)\})-\widehat{H}_k^{[s]}\,,\qquad 0 \leq s \leq m\,,
\]
which corresponds to~\eqref{eq:DFTe} by linearity.

\bigskip
\emph{Step 2:}
On the one hand, we approximate 
$F^{[m+1]}_B(x)$
by evaluating the 
recurrences \eqref{eq:recurrences:mp1} and \eqref{eq:fattenedhal} 
in the grid $x_j$.
On the other hand, we take
\[
H_B^{[m+1]} (x) := 
 \sum_{j=1}^{m+1} \frac{1}{j!}
\frac{\dif^j h^{[m+1-j]} (x+B)}{\dif x^j}\,,
\]
thus obtaining 
\[
\widetilde{e}^{[m+1]}_{B,k} = \mathrm{DFT}_k(\{F^{[m+1]}_B(x_j)\})-\widehat{H}_{B,k}^{[m+1]}\,.
\]

We thus define 
\begin{equation}\label{eq:finalstep2}
C_T:=
\Norm{\tilde e^{[m+1]}_B}_\rho^{\cF}.
\end{equation}

\bigskip
\emph{Step 3:}
We take into account the error produced when approximating $e^{[s]}(x)$
using discrete Fourier approximation.
We control first the term $e^{[0]}(x)-\tilde e^{[0]}(x)$
using Theorem~\ref{theo:DFT}, obtaining
\[
\norm{e^{[0]}-\tilde e^{[0]}}_\rho =
\norm{F^{[0]}-\tilde F^{[0]}}_\rho 
\leq C_N(\rho,\tilde \rho) \norm{F^{[0]}}_{\tilde \rho}
= C_N(\rho,\tilde \rho) \cF_0\,,
\]
where,
for convenience, we have introduced the notation
\[
\cF_0:=\norm{f_{\al_0}(\id+h^{[0]})-\id}_{\tilde\rho}\,.
\]
Notice that $\cF_0<\infty$ due to the assumption in~\eqref{eq:condError}. 
Since $H^{[s]}(x)$ are trigonometric polynomials, we have
\[
\norm{e^{[s]}-\tilde e^{[s]}}_\rho =
\norm{F^{[s]}-\tilde F^{[s]}}_\rho 
\leq C_N(\rho,\tilde \rho) \norm{F^{[s]}}_{\tilde \rho}
\,.
\]
Recalling that $F^{[s]}(x)$ satisfy the recurrences 
\eqref{eq:recurrences}
we obtain
\begin{align}
\norm{F^{[s]}}_{\tilde\rho}={}&
\Norm{\mathfrak{F}_s(
\id+h^{[0]},\al^{[0]};
h^{[1]},\ldots,h^{[s]},
\al^{[0]},\ldots,\al^{[s]},
F^{[1]},\ldots,F^{[s-1]}
)
}_{\tilde\rho} \nonumber \\
\leq {}&
\sup_{x\in \TT_{\tilde\rho}}
\mathfrak{G}_s \Big(
x+h^{[0]}(x),
\al^{[0]};
\norm{h^{[1]}}_{\tilde\rho}^\cF,
\ldots,\norm{h^{[s]}}_{\tilde \rho}^\cF,
\abs{\al^{[1]}},\ldots,\abs{\al^{[s]}}, \label{eq:KK}\\
& \qquad\qquad~\cF_1,\ldots,\cF_{s-1}\Big)
=:\cF_s\,, \nonumber
\end{align}
where the majorant recurrences $\mathfrak{G}_s$ 
are obtained by applying triangular inequalities, Banach algebra properties, and
$\norm{\cdot}_{\tilde\rho}\leq\norm{\cdot}_{\tilde\rho}^{\cF}$ in the
expression of the recurrence $\mathfrak{F}_s$. Notice that the control
of the supremum can be performed using the ideas in Proposition~\ref{prop:H1check}
and optimal bounds are easily obtained for each particular problem at hand.

Similarly, we control the term
$e^{[m+1]}_B(x)-\tilde e^{[m+1]}_B(x)$ as follows
\[
\norm{e_B^{[m+1]}-\tilde e_B^{[m+1]}}_\rho =
\norm{F_B^{[m+1]}-\tilde F_B^{[m+1]}}_\rho 
\leq C_N(\rho,\tilde \rho) \norm{F_B^{[m+1]}}_{\tilde \rho}
\leq C_N(\rho,\tilde \rho) \cF_{B,m+1}
\]
where $\cF_{B,m+1}$ is obtained 
using analogous recurrences
\begin{align}
\cF_{B,s} := {}& \overline{\cG_{B,s}}\,, \nonumber\\
\cG_{B,s} := {}&
\sup_{x\in \TT_{\tilde\rho}}
\mathfrak{G}_s \Big(
x+h_B^{[0]}(x),
\al_B^{[0]};
\norm{h_B^{[1]}}_{\tilde\rho}^\cF,
\ldots,\norm{h_B^{[s]}}_{\tilde \rho}^\cF,
\abs{\al_B^{[1]}},\ldots,\abs{\al_B^{[s]}},\label{eq:KK2} \\
& \qquad\qquad~\cF_{B,1},\ldots,\cF_{B,s-1}\Big) \subset \RR\,, \nonumber
\end{align}
where we used the notation in~\eqref{eq:Intbound} for the right boundary of an interval.
Notice that these recurrences are initialized as
\[
\cF_{B,0}:=\norm{f_{\al_0}(\id+h_B^{[0]})-\id}_{\tilde\rho}\,,
\]
and that the last term is evaluated taking $h^{[m+1]}(x)=0$ and
$\al^{[m+1]}=0$.

We finally define 
\[
C_F:=
\sum_{s=0}^{m} \cF_s \ \mathrm{rad}(B)^s
+
\cF_{B,m+1} \
\mathrm{rad}(B)^{m+1}
\,,
\]
and complete the estimate for $\norm{e}_{\Theta,\rho}$.

\bigskip
\emph{Step 4:}
Since the objects $h(x,\theta)$ and $\al(\theta)$ are polynomials with
respect to $\theta$, we have that $e_\theta(x)$ is smooth in
the domain $B \supset \Theta$. Hence, we can control $\Lip{\Theta,\rho}(e) \leq \norm{\pd_\theta e}_{B,\rho}$
taking derivatives in our Taylor-model. The estimate in the statement follows
directly with the constant $C_F'$ given by
\[
C_F':=
\sum_{s=1}^{m} s \cF_s \ \mathrm{rad}(B)^{s-1}
+
(m+1) \cF_{B,m+1} \
\mathrm{rad}(B)^{m}
\,. \qedhere
\]
\end{proof}

\begin{remark}
A quite technical observation is that along the proof we propose the use of the
Fourier norm to control several trigonometric polynomials
(see~\eqref{eq:finalstep2}, \eqref{eq:KK} and~\eqref{eq:KK2}), rather than
using the estimate $\mathfrak{M}_{B,\rho}(\cdot)$.  The reason is that these
objects have quite large analytic norms so both approaches produce equivalent
estimates, but the advantage of the Fourier norm is that
it is faster to evaluate.  However, in condition~\eqref{eq:condError} we use
$\mathfrak{M}_{\theta_0,\rho}(\cdot)$ instead since it produces a sharper
estimate.
\end{remark}

\section{Application in an example}\label{sec:example:Arnold}

To complement the exposition and the effectiveness of the estimates,
we illustrate the 
performance of our rigorous estimates
with an example.
For a given value of $\ep \in [0,1)$, we 
consider the Arnold family
\begin{equation}\label{eq:Arnold:map}
\al \in [0,1) \longmapsto f_\al(x) = x + \al + \frac{\ep}{2\pi} \sin (2\pi x)
\end{equation}
and 
apply our a-posteriori theorem to obtain effective bounds for the
measure of parameters $\al$ that correspond to conjugacy to rigid rotation.

\bigskip
\emph{Obtaining candidates for $h(x,\theta)$ and $\al(\theta)$:}
Given a fixed rotation number $\theta_0$, the functions $h^{[s]}(x)$
and the numbers $\al^{[s]}$, $0\leq s \leq m$, are
determined by performing Lindstedt-series at the point $\theta=\theta_0$:
\begin{itemize}[leftmargin=5mm]
\item We compute a trigonometric
polynomial $h^{[0]}(x)$ and a constant $\al^{[0]}$,
that approximate the objects $h_0(x) \simeq x+h^{[0]}(x)$
and $\al_0 \simeq \al^{[0]}$ corresponding to the
conjugacy of the map $f_{\al_0}(x)$, 
at a selected value of $\ep$, to a rigid
rotation of angle $\theta_0$. The pair is obtained 
by numerical continuation 
with respect to $\ep\in [0,\ep_0]$ of the
initial objects $h_0(x)=x$ and $\al_0=\theta_0$ that conjugate
the case $\ep=0$
from the case $\ep=0$ (with initial objects $h_0(x)=x$ and $\al_0=\theta_0$).
The interested reader
is referred to \cite{LlaveL11} for implementation details of this
numerical method and to~\cite{CallejaCLa,
CallejaL09, CallejaF, CanadellH17b, HaroCFLM16, HuguetLS12}
for other contexts where numerical algorithms have been designed from
a-posteriori KAM-like theorems.
\item We then compute approximations for the higher order terms of the Lindstedt
series, i.e. trigonometric polynomials $h^{[s]}(x)$ and numbers $\al^{[s]}$,
for $1\leq s\leq m$.
Specifically, we assume inductively that we have computed the exact
Lindstedt series up to order $\ell$
\begin{equation}\label{eq:myaprox}
h_\ell(x,\theta)=x+\sum_{s=0}^{\ell} h^{[s]}(x)(\theta-\theta_0)^s\,, 
\qquad
\al_\ell(\theta)=\sum_{s=0}^{\ell} \al^{[s]}(\theta-\theta_0)^s\,.
\end{equation}
Then, we compute the partial error of conjugacy
\[
e_\ell(x,\theta) 
= 
\sum_{s\geq\ell+1} e_{\ell}^{[s]}(x)(\theta-\theta_0)^s :=
f_{\al_\ell(\theta)}(h_\ell(x,\theta)) - h_\ell(x+\theta,\theta)
\]
using the rules of Automatic Differentiation for the composition with
the sinus function. Then, the terms of order $\ell+1$ satisfy the equation
\[
\pd_x f_{\al_0}(h_0(x)) h^{[\ell+1]}(x)
- h^{[\ell+1]}(x+\theta_0)
+ \pd_\al f_{\al_0} (h_0(x)) \al^{[\ell+1]}
=-e_{\ell}^{[\ell+1]}(x)
\]
which has the same structure as the linearized
equation~\eqref{eq:linear}, so its solutions are
approximated using trigonometric polynomials by
evaluating formulae~\eqref{eq:solc} in Fourier space. 
\end{itemize}
Unless otherwise stated, all computations discussed from now
onwards
are performed using at most $N=2048$ Fourier coefficients,
requesting that the error of invariance at $\theta_0$
satisfies $\norm{e}_0 < 10^{-35}$, and the Lindstedt
series is computed up to order $m=9$.

Following Section~\ref{sec:hypo}, since the objects~\eqref{eq:myaprox}
are polynomials in $\theta$, we denote by $B$ the interval (centered at $\theta_0$) were they
are evaluated.
Regarding the length of the interval $B$, it is clear that 
the application of the
KAM theory fails if $\mathrm{rad}(B)$ is too large, 
so we may need to split the interval
$B$ into subintervals with non-overlapping interior. We carry out
a \emph{branch and bound} procedure, applying the KAM theorem in a subinterval $B_0\subset B$
and repeating the procedure by splitting the set $B\backslash B_0$ into two
smaller intervals. We stop when the intervals are small enough.

Given the numerical approximation described above, the estimates described
in Section~\ref{sec:hypo} are evaluated. In the terminology of validated
computations, if we can apply Theorem~\ref{theo:KAM}
successfully, obtaining explicit control of the pair $h_0(x)$ and $\al_0$, we say
that the numerical computation has been rigorously validated. Moreover,
if we can apply Theorem~\ref{theo:KAM:L} we say that we have validated
a family of conjugacies $h(x,\theta)$ obtaining a rigorous lower bound of the measure
of parameters $\al$ which leads to conjugation.
It is worth mentioning that we are not limited to use Lindstedt series
to obtain a candidate for $h(x,\theta)$ and $\al(\theta)$. 
One should notice (see Sections~\ref{ssec:conjL} and~\ref{sec:hypo}) that the
arguments do not depend on how the candidates are obtained.

\bigskip
\emph{Global constants for hypothesis $\GI$:}
As it was mentioned in Section~\ref{ssec:GIGII}, this part
is problem dependent.
Given $\hat\rho>0$, the derivatives of the Arnold map~\eqref{eq:Arnold:map}
are controlled as
\begin{align*}
& c_x = 1+ \ep \cosh(2 \pi \hat\rho)\,, \\
& c_\al = 1\,, \\
& c_{xx} = 2\pi \ep \cosh(2 \pi \hat\rho)\,, \\
& c_{xxx} = 4 \pi^2 \ep \cosh(2 \pi \hat\rho)\,, \\
& c_{x\al} = c_{\al\al}= c_{x\al\al} = c_{\al\al\al}=0\,.
\end{align*}

\bigskip
\emph{Global constants for hypothesis $\GII$:}
Given an interval of rotation numbers $B\subset (0,1)$ centered at $\theta_0$,
the parameters $\gamma$ and $\tau$
are selected to guarantee that the set of Diophantine numbers in the
set $\Theta=B \cap
\cD(\gamma,\tau)$ reaches a prefixed relative measure.
To this end, we use Lemma~\ref{sec:lemma:Dio} asking for a relative measure of $99\%$.
This fulfills the hypothesis
$\GII$. For example, for an interval with $\mathrm{rad}(B)=1/2^{14}$ centered on
$\theta_0=(\sqrt{5}-1)/2$, such relative measure is achieved taking
$\gamma = 0.0009765625$ and $\tau=1.2$.
Note that, when dividing the interval $B$ into
subintervals, suitable constants $\gamma$ and $\tau$ are recomputed.
Thus, we take into account the resonances
that affect only each subinterval.

\bigskip
\emph{Hypotheses $\HI$, $\HII$ and $\HIII$:}
Using the numerical candidates, we proceed to invoke 
Proposition~\ref{prop:H1check}.
The finite amount of computations are rigorously performed
using computer intervalar arithmetics. For example,
for $\ep=0.25$, $\theta_0=(\sqrt{5}-1)/2$ and $\mathrm{rad}(B)=1/2^{14}$,
we obtain
\begin{align*}
\mathfrak{M}_{B,\rho}(\pd_x h) < {} &             1.16741651 \,, \\
1/\mathfrak{m}_{B,\rho}(\pd_x h) < {} &           1.16623515 \,, \\
\mathfrak{M}_{B,\rho}(\pd_{xx} h) < {} &          1.14231325 \,, \\
\overline{\abs{1/\tilde b_{B,0}}}/(1-c_b) < {} &  0.990541647 \,, \\
\mathfrak{M}_{B,\rho}(\pd_\theta h)< {} &         0.114433851 \,, \\
\mathfrak{M}_{B,\rho}(\pd_{x\theta} h) < {} &     0.842060023 \,, \\
\overline{\al'(B)}  < {} &                        0.990774718 \,, \\
\underline{\al'(B)} > {} &                        0.990772454 \,.
\end{align*}

\bigskip
\emph{Rigorous control of the error of invariance:}
Let us first describe the recurrences $\mathfrak{F}_{s}$ and $\mathfrak{G}_{s}$,
associated to the family~\eqref{eq:Arnold:map}, which are used
to compute the estimate produced in Theorem~\ref{theo:Finite}.
Given series
\[
h(x,\theta)=x+\sum_{s=0}^m h^{[s]}(x) (\theta-\theta_0)^s
\,, \qquad
\al(\theta)=\sum_{s=0}^m \al^{[s]} (\theta-\theta_0)^s
\]
we have that the coefficients of
\[
F(x,\theta)=f(h(x,\theta),\al(\theta))= x+\sum_{s\geq 0} F^{[s]}(x) (\theta-\theta_0)^s
\,,
\]
are obtained using the formula $\mathfrak{F}_s$, which in this case corresponds
to evaluate the recurrences
\[
F^{[s]}(x) =
        h^{[s]}(x) + \al^{[s]} + \frac{\ep}{2\pi} \Sin^{[s]}(x)\,,
\]
where
\begin{align*}
&\Sin^{[0]}(x) = \sin(2\pi (x+h^{[0]}(x)))
\,,\\
&\Cos^{[0]}(x) = \cos(2\pi (x+h^{[0]}(x)))\,,
\end{align*}
and, for $s\geq 1$,
\begin{align*}
& \Sin^{[s]}(x) = \frac{2\pi}{s} \sum_{j=0}^{s-1} (s-j)h^{[s-j]}(x) \Cos^{[j]}(x)
\,, \\
& \Cos^{[s]}(x) = -\frac{2\pi}{s} \sum_{j=0}^{s-1} (s-j)h^{[s-j]}(x) \Sin^{[j]}(x)\,.
\end{align*}

When the time comes to control the error produced with Fourier discretization,
these formulae lead to (in this case the functions $\sin(\cdot)$ and $\cos(\cdot)$
have the same bounds)
\[
\norm{F^{[s]}}_{\tilde\rho} \leq 
\norm{h^{[s]}}_{\tilde\rho}+ \abs{\al^{[s]}}+\frac{\ep}{2\pi} \cS_s 
= \cF_{s}\,,
\]
for $s\geq 0$, where the constants $\cS_s$ are initialized as
\[
\cS_0 = \cosh(2\pi(\tilde\rho+\norm{h^{[0]}}_{\tilde\rho}^{\cF}))
\]
and then obtained recursively, for $s\geq 1$, as
\[
\cS_s = \frac{2 \pi}{s} \sum_{j=0}^{s-1}
(s-j) \norm{h^{[s-j]}}_{\tilde\rho}^\cF \ \cS_{j}\,.
\]
Analogous formulae are used to compute $\cF_{B,m+1}$.

\bigskip
\emph{Selection of the remaining KAM parameters:}
Finally, there is a set of parameters that must be selected in order
the apply the a-posteriori theorems: $\rho$, $\delta$, $\rho_\infty$,
$\hat \rho$, $\tilde \rho$ and parameters to control
the initial objects.
To choose the parameters 
$\sigma_1$,
$\sigma_2$, $\sigma_3$, $\sigma_b$, $\beta_0$, $\beta_1$ and $\beta_2$, we
use a single parameter $\sigma>1$ together with the sharp estimates produced
in Proposition~\ref{prop:H1check}.
For example, we take $\sigma_1=\mathfrak{M}_{B,\rho}(\pd_x h) \sigma$.
Suitable KAM parameters are selected following a
heuristic procedure, adapted mutatis mutandis from~\cite[Appendix A]{FiguerasHL17},
that allows us to optimize the values of the
constants $\mathfrak{C}_1$, $\mathfrak{C}_2$, $\mathfrak{C}_1^\sLip$,
and $\mathfrak{C}_2^\sLip$. 

\bigskip
\emph{Some specific results:}
In the above scenario, for $\ep=0.25$ and $B = (0,1)$, we
choose the set of Diophantine numbers $\Theta \subset B$ such that
$\meas(\Theta)>0.99$. After applying Theorem~\ref{theo:KAM:L}
we obtain a lower bound
$\meas(\al_\infty(\Theta)) > 0.860748$ for the absolute measure of
parameters which correspond to rotation.
To see how sharp is this estimate, we compute also
a lower bound of the measure of the phase-locking intervals, thus obtaining
\[
0.085839 < \meas([0,1]\backslash \al_\infty(\Theta))\,.
\]
This lower bound follows by computing (using standard rigorous Newton method)
two $p/q$-periodic orbits
close to the boundaries of each interval of rotation $p/q \in \QQ$, for $q\leq 20$.
To summarize, in terms of the notation in the introduction, we have
\[
0.860748 < \meas(K_{0.25}) < 0.914161\,.
\]
Notice that, assuming that the measure was $\meas(K_{0.25}) \simeq 0.914161$, our
rigorous lower bounds predicts the 94.15\% of the measure.

Of course, most part of underestimation
corresponds to the resonances $0/1$, $1/2$, $1/3$, $2/3$, $1/4$ and
$3/4$. If we remove the set of subintervals (of total measure $0.082046$) 
were we fail to apply the theorem
(mostly around the mentioned resonances), then it turns out that the produced
lower bound corresponds to a relative measure of $93.76$\% in the
considered set of rotations $B$ (which now has measure $0.917954$).
Indeed, for the interval $B = (391/1024,392/1024)$
(which contains $(3-\sqrt{5})/2$) we obtain a 
relative lower bound of $98.26$\% for the existence of conjugacies. These lower
bounds can be improved by increasing the number of Fourier coefficients and
the tolerances, but these numbers serve as an illustration with moderate
computational effort. For example, using a single desktop computer,
the validation of the interval $B =
(391/1024,392/1024)$ takes around 
$10$ minutes
(49 subdivisions of $B$ are required)
and the validation of the
interval $B=(0,1)$ takes 
around $30$ days
(161891 subdivisions of $B$ are required). As
expected, the bottleneck are resonant rotation numbers which do not contribute
in practice.

In Table~\ref{tab:table:results} we illustrate how the lower bound depends on
the selected value of $\ep_0$. In order to illustrate the computational
cost of the validation, we show the largest interval $B$ (of length of
the form $1/2^{n}$) centered at the golden rotation number such the theorem
can be applied without subdividing the interval. This, together with the number
of Fourier coefficients required, illustrates the technological difficulty
to apply the KAM theory when one approaches the critical limit. 

\begin{table}[!h]
\centering
\begin{tabular}{|l| r c |}
\hline
$\qquad\qquad\ep_0$ & $N$ & $\meas(\al_\infty(\Theta))/\meas(B)>$ \\
\hline
$\hphantom{00}1/2^7 = 0.0078125$ & $64$ & $0.999533$ ($\meas(B)=1/2^{12}$)\\
$\hphantom{0}10/2^7 = 0.078125$ & $128$ & $0.998661$ ($\meas(B)=1/2^{12}$)\\
$\hphantom{0}20/2^7 = 0.15625$ & $256$ &  $0.995996$ ($\meas(B)=1/2^{14}$)\\
$\hphantom{0}30/2^7 = 0.234375$ & $256$ & $0.991461$ ($\meas(B)=1/2^{14}$)\\
$\hphantom{0}40/2^7 = 0.3125$ & $256$ &   $0.984921$ ($\meas(B)=1/2^{15}$)\\
$\hphantom{0}50/2^7 = 0.390625$ & $512$ & $0.976080$ ($\meas(B)=1/2^{17}$)\\
$\hphantom{0}60/2^7 = 0.46875$ & $512$ &  $0.964547$ ($\meas(B)=1/2^{17}$)\\
$\hphantom{0}70/2^7 = 0.546875$ & $512$ & $0.949686$ ($\meas(B)=1/2^{18}$)\\
$\hphantom{0}80/2^7 = 0.625$ & $1024$ &   $0.930482$ ($\meas(B)=1/2^{19}$)\\
$\hphantom{0}90/2^7 = 0.703125$ & $1024$ &$0.905233$ ($\meas(B)=1/2^{19}$)\\
$100/2^7 = 0.78125$ & $1024$ &            $0.870752$ ($\meas(B)=1/2^{20}$)\\
$110/2^7 = 0.859375$ & $2048$ &           $0.819862$ ($\meas(B)=1/2^{22}$)\\
$120/2^7 = 0.9375$ & $4096$ &             $0.728697$ ($\meas(B)=1/2^{25}$)\\
$123/2^7 = 0.9609375$ & $8192$ &          $0.678925$ ($\meas(B)=1/2^{26}$)\\
$125/2^7 = 0.9765625$ & $16384$ &         $0.627992$ ($\meas(B)=1/2^{28}$)\\
\hline
\end{tabular}
\caption{ {\footnotesize
Rigorous lower bounds for the measure of the conjugacy of rotations
for the Arnold map (for different values of $\ep_0$) corresponding to
a small interval $B$ of rotation numbers centered at
$\theta_0 = (\sqrt{5}-1)/2$. For each $\ep_0$ we choose the largest
interval that allows us to apply Theorem~\ref{theo:KAM:L} without
subdividing it.
We include also the number of
Fourier coefficients used.}} \label{tab:table:results}
\end{table}

For the sake of completeness, we finally include a list with some parameters
and intermediate estimates associated to the computer-assisted proof
corresponding to $\ep_0=0.25$ and an
interval with $\mathrm{rad}(B)=1/2^{14}$ centered on
$\theta_0=(\sqrt{5}-1)/2$.
\begin{align*}
\rho       ={} & 1.060779991992726 \cdot 10^{-2}\,, \\
\rho_{\infty}={} & 1.060779991992726 \cdot 10^{-5}\,, \\
\delta     ={} & 2.651949979981816 \cdot 10^{-3}\,, \\
\hat\rho   ={} & 3.352069177291399 \cdot 10^{-2}\,, \\
\tilde\rho ={} & 1.272935990391272 \cdot 10^{-1}\,, \\
\sigma     ={} & 1.000107420067662 \,, \\
C_T        ={} & 5.898764259376722 \cdot 10^{17}\,, \\
C_F        ={} & 7.510700397297086 \cdot 10^{-1}\,, \\
\norm{e}_{\Theta,\rho} 
< {}&                      5.555151225281469 \cdot 10^{-24} \,, \\
\Lip{\Theta,\rho}(e)
< {}&                      9.101559767500465 \cdot 10^{-19} \,, \\
\kappa <{} &               3.000359067414863 \cdot 10^{-11} \,, \\
\mu <{} &                  5.097018817302322 \cdot 10^{-11} \,, \\
\mathfrak{C}_1^\sLip <{} & 5.364258457375897 \cdot 10^{5}\,, \\
\mathfrak{C}_2^\sLip <{} & 2.440549276176043 \cdot 10^{1}\,, \\
\tfrac{\mathfrak{C}_1^\sLip}{\gamma^{2} \rho^{2\tau+2}}\cE
< {}&                      6.778257281684347 \cdot 10^{-1} \,, \\
\tfrac{\mathfrak{C}_2^\sLip}{\gamma \rho^{\tau+1}}\cE
< {}&                      6.466059566862642 \cdot 10^{-14}\,,
\end{align*}
where we recall that
\[
\cE:=\max \left\{ \norm{e}_{\Theta,\rho} \,,\,\gamma \delta^{\tau+1}\Lip{\Theta,\rho}(e)\right\}\,.
\]
Then, we apply Theorem~\ref{theo:KAM:L} and obtain a lower bound
$\meas(\al_\infty(\Theta))>0.000120024$ of the absolute measure, which corresponds to a relative measure of $98.32\%$
in the selected interval.

\section{Final remarks}\label{sec:remarks}

To finish we include some comments regarding
direct applications and generalizations of the results presented in the paper.
Our aim is to present a global picture of our approach and to
establish connections with different contexts.

\bigskip
\emph{Asymptotic estimates in the local reduction case}.
Although Theorem~\ref{theo:KAM:L} has been developed with the aim
of performing computer-assisted applications in non-perturbative regimes,
it is clear it allows recovering the perturbative setting.
Indeed, as a direct corollary, we obtain asymptotic measures (\`a la Arnold)
of conjugacies close to rigid rotation in large regions of parameters. 

\begin{corollary}\label{coro:perturbado}
Consider a family of the form
\[
f_\al(x)=x+\al+\ep g(x)\,,
\]
with $g \in \Per{\hat\rho}$, $\hat\rho>0$. 
Then there is constants $\mathfrak{C}_1^\sLip$, $\mathfrak{C}_2^\sLip$
(which are directly computed using Theorem~\ref{theo:KAM:L}) such that if
$\ep$ satisfies
\[
\ep \frac{\mathfrak{C}_1^\sLip \norm{g}_\rho}{\gamma^2 \delta^{2\tau+2}} < 1\,,
\]
with $\rho<\hat\rho$, $\gamma<1/2$ and $\tau>0$,
then we have
\begin{align*}
\meas (\al_\infty(\Theta)) \geq {} & 
\left(
1
- 
\ep \frac{\mathfrak{C}_2^\sLip \norm{g}_\rho}{\gamma \rho^{\tau+1}}
\right)
\left(1-2\gamma \frac{\zeta(\tau)}{\zeta(\tau+1)}\right) \\
> {} & 1-
\gamma
\left(
\frac{\mathfrak{C}_2^\sLip \rho^{\tau+1}}{\mathfrak{C}_1^\sLip}
- 2 \frac{\zeta(\tau)}{\zeta(\tau+1)}
\right) + \cO(\gamma^2)\,, 
\end{align*}
where $\Theta=[0,1]\cap \cD(\gamma,\tau)$ and $\zeta$ is the Riemann zeta function.
\end{corollary}

\begin{proof}
We consider the candidates
\[
h(x,\theta)=x\,, \qquad \al(\theta)=\theta\,,
\]
and apply Theorem~\ref{theo:KAM:L}. First, we notice
that the error of conjugacy is of the form
$e(x)=\ep g(x)$, independent of $\theta$, so
we have the estimates
\[
\norm{e}_{\Theta,\rho} \leq \ep \norm{g}_\rho\,,
\qquad
\Lip{\Theta,\rho}(e)=0\,.
\]
To satisfy the hypothesis $\HI$, $\HII$, and $\HIII$ we introduce
an uniform parameter $\sigma>1$ and we take
\begin{align*}
& \sigma_1=\sigma_2=\sigma_b=\beta_2=\sigma>1\,,\\
& \sigma_3=\beta_0=\beta_1=\sigma-1>0\,.
\end{align*}
Then, we compute the constant $\mathfrak{C}_1^\sLip$ and
assume that $\ep$ is small enough so that the
condition~\eqref{eq:KAM:C1:L} holds. 
Indeed, the largest value of $\ep$ that saturates this condition
can be obtained by selecting a suitable value of $\sigma$.
Then, the measure of parameters is given by the formula
\[
\meas (\al_\infty(\Theta)) \geq
\left[
1
- 
\frac{\mathfrak{C}_2^\sLip
\ep \norm{g}_{\rho}}{\gamma \rho^{\tau+1}}
\right]
\meas(\Theta)\,.
\]
The statement follows recalling
the estimate for $\meas(\Theta)$ in Remark~\ref{rem:meas:gbl}.
\end{proof}

\bigskip
\emph{Conjugation of maps on $\TT^d$}.
Theorem~\ref{theo:KAM} and Theorem~\ref{theo:KAM:L} can be readily
extended to consider a family $\al \in A \subset \RR^d \mapsto f_\al \in \Anal(\TT^d_\rho)$.
In this case, a map $g : \TT^d \rightarrow \TT^d$ is viewed as a vector, and
the norm $\norm{g}_{\rho}$, for $g\in \Per{\rho}$, is taken as the induced norm. Then, all the
arguments and computations are extended
to matrix object with no special difficulty. For example, the torsion
matrix becomes
\[
b(x)= \Dif_x h(x+\theta)^{-1} \Dif_\al f(h(x))
\]
where $\theta \in \RR^d$ is the selected Diophantine vector.
The interested reader is referred to~\cite{CanadellH17a} for details regarding the
corresponding Theorem~\ref{theo:KAM} and to~\cite{FiguerasHL17} for details
regarding the approximation of functions in $\TT^d$ using discrete Fourier transform.

\bigskip
\emph{Other KAM contexts}.
We have paid special attention to present Theorem~\ref{theo:KAM} and Theorem~\ref{theo:KAM:L}
separately. The important message is that our methodology can be readily
extended to any problem, as long as there exists an a-posteriori KAM theorem
(which replaces Theorem~\ref{theo:KAM}) for the existence of quasi-periodic dynamics:
Lagrangian tori
in Hamiltonian systems or symplectic maps \cite{GonzalezJLV05,HaroCFLM16,HaroL},
dissipative systems \cite{CallejaCLa},
skew-product systems \cite{FiguerasH12,HaroCFLM16,HaroL06a}
or lower dimensional tori \cite{FontichLS09,LuqueV11}, just to mention a few.
For each of such theorems, a judicious revision of the corresponding KAM scheme must be performed
in order to obtain sharp estimates of the Lipschitz dependence on parameters
(which replaces Theorem~\ref{theo:KAM:L}). It is worth mentioning that we have
devoted a significant effort to explain, in common analytic terms, the technical
issues that permits the computer-assisted validation of the hypothesis of our theorems.
This is an important step that prevents the computer-assisted proof to be a ``black
box'' full of tricks that can be only appreciated by a few experts.

\bigskip
\emph{Rigorous validation of rotation numbers}.
Last but not least, another significant corollary of the methodology developed in this
paper is that it allows to rigorously enclose the
rotation number of a circle map. Due to the relevance of this topological
invariant, during the last years, many numerical methods
have been developed for this purpose.
We refer for example to the works~\cite{Bruin92,LaskarFC92,LuqueV08,LuqueV09,Pavani95,SearaV06}
and to~\cite{DasSSY17} for a remarkable method with infinite order convergence,
based on Birkhoff averages. As illustrated in Section~\ref{sec:example:Arnold},
the KAM theorems presented in this paper allows us to obtain a rigorous
enclosure (as tight as required) for the rotation number of a map, which is
interesting in order to rigorously validate the numerical approximations performed
with any of the mentioned numerical methods.

\section*{Acknowledgements}

J.-Ll. F. acknowledges the partial support from Essen, L. and C.-G., for
mathematical studies.  A. H. is supported by the grants MTM2015-67724-P
(MINECO, FEDER), MDM-2014-0445 (MINECO) and 2014 SGR 1145 (AGAUR), and the
European Union's Horizon 2020 research and innovation programme MSCA 734557.
A.L.\ is supported by the Knut och Alice Wallenbergs stiftelse KAW 2015.0365.
We acknowledge computational resources at Uppsala Multidisciplinary Center for
Advanced Computational Science (UPPMAX) (projects SNIC 2018/8-37 and SNIC
2018/8-38).  We acknowledge the use of the Lovelace (ICMAT-CSIC) cluster for
research computing, co-financed by the Spanish Ministry of Economy and
Competitiveness (MINECO), European FEDER funds and the Severo Ochoa Programme.

\bibliographystyle{plain}
\bibliography{references}

\end{document}